\documentclass[12pt]{amsart}

\usepackage[initials]{amsrefs}
\usepackage{amssymb,amsmath,amsfonts,latexsym,amsthm,graphicx} %paralist,

\usepackage[a4paper,nomarginpar]{geometry}

\usepackage[english]{babel}

\usepackage[latin1]{inputenc}

\newtheorem{theorem}{Theorem}[section]

\newtheorem{proposition}[theorem]{Proposition}

\theoremstyle{definition}

\newtheorem{remark}[theorem]{Remark}
\newtheorem{remarks}[theorem]{Remarks}
\newtheorem{definition}[theorem]{Definition}

\input epsf

\newcommand{\N}{\mathbb{N}} %% Conjunto naturales:     \N
\newcommand{\R}{\mathbb{R}} %% Conjunto reales:        \R
\newcommand{\C}{\mathbb{C}} %% Conjunto complejos:     \C
\newcommand{\D}{\mathbb{D}} %% Disco unidad:           \D
\newcommand{\T}{\mathbb{T}} %% Circ. unidad:           \T
\newcommand{\K}{\mathbb{K}}
\newcommand{\ovl}{\overline}
\newcommand{\ve}{\varepsilon}
\newcommand{\la}{\lambda}
\newcommand{\dis}{\displaystyle}

\newcommand{\id}{\mbox{id}}

\begin{document}

\title[Simultaneous universality]{Simultaneous universality}

\date{}

\author[Bernal]{L.~Bernal-Gonz\'alez}
\address{Departamento de An\'{a}lisis Matem\'{a}tico
\newline\indent Facultad de Matem\'{a}ticas, \newline\indent Universidad de Sevilla
\newline\indent Apdo. 1160, Avenida Reina Mercedes
\newline\indent 41080 Sevilla, Spain.}
\email{lbernal@us.es}

\author[Jung]{A.~Jung}
\address{Fachbereich IV Mathematik
\newline\indent Universit\"at Trier
\newline\indent D-54286 Trier, Germany}
\email{s4anjung@uni-trier.de}


\thanks{}

\begin{abstract}
In this paper, the notion of simultaneous universality is introduced, concerning operators having orbits that simultaneously approximate any given vector.
This notion is related to the well known concepts of universality and disjoint
universality. Several criteria are provided, and several applications to specific operators or sequences of operators are performed, mainly in the
setting of sequence spaces or spaces of holomorphic functions.
\end{abstract}

\maketitle

\section{Introduction}

In this paper, we are concerned with the phenomenon of simultaneous appro\-xi\-ma\-tion by the action of several
operators or, more generally, by the action of several sequences of mappings.
When the existence of a dense orbit under an operator is proved,
we are speaking about universality or hypercyclicity, see below.
In many situations, it is possible to show the existence of one vector whose orbits under two or more operators
approximate any given vector. Pushing the question quite further, we wonder under what conditions such
approximation takes place by using a {\it common} subsequence. This, together with its connection with
other kinds of joint universality, will make up the main aim of the present manuscript.

\vskip .15cm

Next, we fix some related notation and terminology to be used in this work.
For a good account of concepts, results and history concerning
hypercyclicity, the reader is referred to the books \cite{bayartmatheron2009,grosseperis2011}.

\vskip .15cm

By $\N , \, \N_0 , \, \R , \, \C , \, \D , \, B(a,r), \,
\ovl{B}(a,r)$ $(a \in \C , \, r > 0)$ we denote, respectively, the
set of positive integers, the set $\N \cup \{0\}$, the real line,
the complex plane, the open unit disk $\{z \in \C : \, |z| < 1\}$,
the open disk with center $a$ and radius $r$, and the
corresponding closed disk. Let $X,Y$ be two Hausdorff topological spaces, and
$T_n:X \to Y$ $(n=1,2, \dots )$ be a sequence of continuous mappings. Recall that
$(T_n)$ is said to be {\it universal} whenever there is some $(T_n)$-orbit which is dense in $Y$, that is,
there exists an element $x_0 \in X$ --called universal for $(T_n)$-- such that
$$\overline{\{T_nx_0: \, n \in \N\}} = Y.$$
Note that $Y$ must be separable. We denote by \,$\mathcal U ((T_n))$ \,the set of universal ele\-ments for $(T_n)$.
When $X = Y$ and $T:X \to X$ is a continuous self-mapping, then $T$ is called {\it universal}
provided that the sequence $(T^n)$ of iterates of $T$ (i.e., $T^1 = T$, $T^2 = T \circ T$, $T^3 = T \circ T^2$, and so on)
is universal, in which case the set $\mathcal U ((T^n))$ of universal elements will be denoted by $\mathcal U (T)$.
A sequence $T_n:X \to Y$ $(n=1,2, \dots )$ of continuous mappings is said to be {\it densely universal} \,if \,$\mathcal U ((T_n))$
is dense in $X$. Birkhoff's transitivity theorem asserts that, if $X$ is a Baire space (in particular, if $X$ is completely metrizable)
and $Y$ is second-countable (in particular, if $X$ is metrizable and separable), then $(T_n)$ is densely universal if and only if
$(T_n)$ is {\it transitive} (that is, given nonempty open sets $U \subset X$, $V \subset Y$, there is $N \in \N$ with $T_N(U) \cap V \ne \varnothing$);
if this is the case, then $\mathcal U ((T_n))$ is residual (in fact, a dense $G_\delta$ subset) in $X$. If $X$ lacks isolated points and $T:X \to X$ is universal, then $\mathcal U (T)$ is dense in $X$ (so residual if $X$ is, in addition, completely metrizable).

\vskip .15cm

In the case in which $X$ and $Y$ are topological vector spaces over $\K$ ($=\R$ or $\C$) and $(T_n) \subset L(X,Y) := \{$linear continuous mappings $X \to Y\}$, the words {\it hypercyclic} and {\it universal} are synonymous, although {\it hypercyclic} is mostly used, as well as the alternative notation
$HC((T_n)) := \mathcal U ((T_n))$ (and $HC(T) := \mathcal U (T)$ for $T \in L(X) := L(X,X) = \{$operators on $X\}$).
In particular, we have if $X$ and $Y$ are F-spaces with $Y$ separable,
then $HC((T_n))$ ($HC(T)$, with $X$ separable, resp.) is residual in $X$ as soon as $(T_n)$ is transitive (as soon as $T$ is hypercyclic, resp.).
Recall that an F-space is a completely metrizable topological vector space.

\vskip .15cm

Assume now that \,$X,\, Y$ \,are topological spaces, with $X$ a Baire space and $Y$ second-countable, and that
\,$S_n:X \to Y$ \,and \,$T_n:X \to Y$ $(n \in \N )$ are densely
universal sequences. Since \,$\mathcal U ((S_n)), \, \mathcal U ((T_n))$ are dense $G_\delta$ subsets of $X$,
we have that $\mathcal U ((S_n)) \cap \mathcal U ((T_n))$
is also dense, so non-empty. Hence there is a common hypercyclic
element $x \in X$. So, for a given point $y \in Y$,
there are sequences $\{n_1 < n_2 < \cdots \}$ and $\{m_1 < m_2 <
\cdots \}$ in $\N$ such that
$$
S_{n_j}x \to y \hbox{ \ and \ } T_{m_j}x \to y \hbox{ \ as \ } j \to \infty .
$$
Then the following question arises naturally:

\vskip .15cm

{\noindent \leftskip 1.2cm \rightskip 1.2cm {\it Under what conditions
on \,$(S_n)$ \,and \,$(T_n)$ \,one can guarantee the exis\-ten\-ce of an element
\,$x \in X$ \,such that, for any given \,$y \in Y$, there
is \,{\rm one} sequence $\{n_1 < n_2 < \cdots \} \subset \N$ such
that
$$S_{n_j}x \longrightarrow \, y \, \longleftarrow T_{n_j}x \hbox{ \ as \ } j \to \infty ?$$} \par}
\vskip -.35cm
Of course, a similar question can be
posed for finitely many sequences and for finitely many single operators on $X$, just by considering
the sequences of their iterates in the latter case. With this in mind, the new concept of
simultaneous universality will be introduced in the next section, and compared to other related notions
existing in the literature, such as
those of disjoint hypercyclicity and the weakly mixing property. Several sufficient conditions for simultaneous universality/hypercyclicity
will be provided in Section 3. Examples of finite families of simultaneous hypercyclic operators will be furnished in sections 4--6,
starting with multiples of an operator and ending up in the frameworks of sequence spaces and of spaces of analytic functions on complex domains.

\section{Simultaneously universal sequences}

\quad Let us define the new concept that is the matter of this
paper. If $p \in \N$ and $Y$ is a nonempty set, then by $\Delta (Y^p)$ we denote the diagonal
of $Y^p = Y \times \cdots \times Y$ ($p$ times), that is, the subset $\Delta (Y^p) = \{(y,y, \dots ,y): \, y \in Y\}$.
If $Y$ is a topological space, then $Y^p$ is assumed to be endowed with the product topology.

\begin{definition}  \label{maindefinition}
{\rm
Let $p \in \N$ and $X,Y$ be Hausdorff topological
spaces. Assume that, for each $j \in \{1, \dots ,p\}$, $T_{j,n}:X
\to Y$ $(n \in \N )$ is a sequence of  continuous
mappings. Consider the sequence
$$[T_{1,n}, \dots ,T_{p,n}]: x \in X \, \longmapsto \, (T_{1,n}x, \dots ,T_{p,n}x) \in Y^p \quad (n \in \N ).$$
Let also \,$T_1, \dots ,T_p : X \to X$ \,be continuous mappings.
\begin{enumerate}
\item [\rm (a)] We say that the sequences $(T_{1,n}), \dots , (T_{p,n})$
are {\it simultaneously universal} \break
(or {\it s-universal}\,) \,whenever
there exists an element $x_0 \in X$ --called {\it s-universal} for $(T_{1,n}), \dots , (T_{p,n})$-- satisfying
$$\ovl{\{[T_{1,n}, \dots ,T_{p,n}]x_0: \, n \in \N\}} \, \supset \, \Delta (Y^p).$$
The set of such s-universal elements will be denoted by $s$-$\mathcal U ((T_{1,n}), \dots , (T_{p,n}))$.
\item [\rm (b)] The sequences $(T_{1,n}), \dots ,(T_{p,n})$ are said to be
{\it densely simultaneously
universal} if the set $s$-$\mathcal U ((T_{1,n}), \dots , (T_{p,n}))$ is
dense in $X$. And they are called {\it hereditarily simultaneously universal}
({\it hereditarily densely simultaneously universal}, resp.) if, for every strictly increasing sequence $(n_k) \subset \N$, the sequences
$(T_{1,n_k}), \dots ,(T_{p,n_k})$ are s-universal (densely s-universal, resp.).
\item [\rm (c)] The mappings \,$T_1, \dots ,T_p$ \,are called {\it s-universal}
({\it densely s-universal}, {\it he\-re\-di\-ta\-ri\-ly s-universal},
{\it hereditarily densely s-universal}, resp.) if the sequences
$(T_1^n), \dots ,(T_p^n)$ are s-universal (densely s-universal, hereditarily s-universal, hereditarily densely s-universal, resp.).
The set $s$-$\mathcal U ((T_1^n), \dots , (T_p^n))$ of corresponding s-universal elements will be denoted by
$s$-$\mathcal U (T_1, \dots , T_p)$.
\end{enumerate}
}
\end{definition}

\begin{remarks}
\noindent 1. If $Y$ is first-countable
(in particular, if $Y$ is metrizable), then the s-simultaneous universality of
$(T_{j,n})_{n\in\mathbb{N}}$ $(1 \le j \le p)$ means the existence of some $x_0 \in X$ enjoying the property that, for every $y \in Y$, there is
a (strictly increasing) sequence $(n_k) \subset \N$ such that \,$T_{j,n_k} x_0 \to y$ \,as \,$k \to \infty$ \,$(j=1, \dots ,p)$.

\vskip 5pt

\noindent 2. In \cite[Kapitel 1]{grosse1987} the notion of relative universality on a closed subset of the arrival space is introduced
under very general assumptions. In the present paper we study a special case of this situation (note that $\Delta (Y^p)$ is closed in $Y^p$ since $Y^p$ is Hausdorff) under more specific hypotheses.

\vskip 5pt

%\noindent 3. Obviously, the s-universality of $(T_{1,n}),
%\dots , (T_{p,n})$ remains unaltered under a change of order
%$(T_{\sigma (1),n}), \dots , (T_{\sigma (p),n})$.
%
%\vskip 5pt

\noindent 3. According to the introduction, if $X,Y$ are to\-po\-lo\-gi\-cal vector spaces and
$T_{j,n},T_j \in L(X,Y)$ $(j = 1, \dots ,p; \, n \in \N )$, then we use the expressions ``s-hypercyclic'', ``densely s-hypercyclic'' and ``hereditarily densely s-hypercyclic'' rather than ``s-universal'',
``densely s-universal'' and ``hereditarily densely s-universal'', respectively. In addition,
we will denote $s$-$HC ((T_{1,n}), \dots , (T_{p,n})) :=$ \break
$s$-$\mathcal U ((T_{1,n}), \dots , (T_{p,n}))$ \,and
\,$s$-$HC (T_1, \dots , T_p) :=$ $s$-$\mathcal U (T_1, \dots , T_p)$ in this case.

\vskip 5pt

\noindent 4. For a single operator $T$, hypercyclicity (hereditary hypercyclicity, resp.) is equivalent to dense hypercyclicity
(hereditary dense hypercyclicity, resp.).

\vskip 5pt

\noindent 5. The property of simultaneous universality of $(T_{1,n}), \dots , (T_{p,n})$ is weaker than the property that the sequence
$([T_{1,n}, \dots ,T_{p,n}])$ is \textit{subspace-universal} for $\Delta (Y^p)$, meaning that the set
$\{[T_{1,n}, \dots ,T_{p,n}]x_0: \, n \in \N\} \cap \Delta (Y^p)$ is dense in $\Delta (Y^p)$ for some $x_0 \in X$
(see e.g.\,\cite{bamernikadetskilicman2016,le2011,madoremartinez2011} for results on subspace-hypercyclicity/universality).
\end{remarks}

Before going on, we want to compare s-universality to other related concepts defined in the
literature. In 2007, B\`es, Peris and the first author (\cite{besperis2007},\cite{bernal2007}) introduced the notion of disjoint (or d-) universality
(sometimes called d-hypercyclicity in the mentioned references). Under the same assumptions and terminology as in
Definition \ref{maindefinition}, the sequences $(T_{1,n}), \dots , (T_{p,n})$ are said to be
\textit{d-universal} whenever the sequence $[T_{1,n} \dots ,T_{p,n}]: X \to Y^p$ $(n\in\mathbb{N})$ is universal, that is, whenever there exists some $x_0\in X$ such that the
joint orbit $\{(T_{1,n}x_0, \dots ,T_{p,n}x_0):n\in\mathbb{N}\}$ is dense in $Y^p$. As a matter of fact, d-universality should not be confused with the universality of the sequence
$$
T_{1,n} \oplus \cdots \oplus T_{p,n}: (x_1, \dots ,x_p) \in X^p
\longmapsto (T_{1,n}x_1, \dots ,T_{p,n}x_p) \in Y^p.
$$
Trivially, disjoint universality of $(T_{1,n}), \dots , (T_{p,n})$ implies universality of the last sequence as well as simultaneous universality of $(T_{1,n}), \dots , (T_{p,n})$. Also, trivially, s-universality implies the universality of each sequence $(T_{j,n})_{n\in\mathbb{N}}$ $(j=1, \dots ,p)$ (in particular, $Y$ must be separable). But no other implications among these properties hold, even considering only $p=2$ and sequences of iterates of single operators. The following examples illustrate this situation:

\begin{enumerate}
\item [1.] Assume that $T$ is a hypercyclic operator on a topological vector space. Then the operators $T,T$ are s-hypercyclic but not d-hypercyclic.
\item [2.] In 1969, S.~Rolewicz \cite{rolewicz1969} proved that if
$c \in \K$ has modulus $>1$ and $B$ is the backward shift $(x_n) \in \ell_2 \mapsto (x_{n+1}) \in \ell_2$, then the
operator $cB$ is hypercyclic. In particular, the operators $T = 2B$ and $S = 4B = 2T$ are hypercyclic, but $T,S$ are clearly not s-hypercyclic.
\item [3.] Since each of the operators $T,S$ of the latter example is mixing (see the definition at the beginning of the next section, regarding
the sequences of iterates; see also \cite[p.~46]{grosseperis2011}), the operator $T \oplus S$ is hypercyclic, but $T,S$ are not s-hypercyclic.
\item [4.] De la Rosa and Read \cite{delarosaread2009} were able to construct a Banach space $X$ and an operator $T \in L(X)$ such that
$T$ is hypercyclic (hence $T,T$ are s-hypercyclic) but $T$ is not weakly mixing on $X$, meaning that $T \oplus T$ is not hypercyclic on $X^2$.
\end{enumerate}

While d-hypercyclic operators must be substantially different, s-hypercyclicity allows more similarity.
For instance, an operator can never be d-hypercyclic with a scalar multiple of itself (see \cite[p.~299]{besperis2007}).
Nevertheless, s-hypercyclicity is possible in concrete situations. This will be analyzed in Section 4.
Sections 5 and 6 are devoted to more specific operators, namely backward shifts and operators on spaces of analytic functions.

\vskip .15cm

We close this section by establishing, under appropriate assumptions, the existence of large vector subspaces
consisting, except for zero, of s-hypercyclic vectors.

\begin{theorem}
\begin{enumerate}
\item[\rm (a)] Let $X$ be a topological vector space and \,$T_j \in L(X)$ $(j=1, \dots ,p)$.
If \,$T_1, \dots ,T_p$ \,are s-hypercyclic and at least one of them commutes with the others, then
$s$-$HC(T_1, \dots ,T_p)$ contains, except for \,$0$, a dense linear subspace of $X$.
\item[\rm (b)] Let $X$ and $Y$ be two topological vector spaces such that \,$Y$ \,is metrizable.
Assume that $(T_{j,n})\subset L(X,Y)$ $(j=1, \dots ,p)$ are hereditarily s-hypercyclic sequences.
Then $s$-$HC((T_{1,n}), \dots ,(T_{p,n}))$ contains, except for \,$0$, an infinite dimensional vector subspace of $X$. %_{n \ge 1}
\item[\rm (c)] Let $X$ and $Y$ be two metrizable separable topological vector spaces.
Assume that $(T_{j,n})\subset L(X,Y)$ $(j=1, \dots ,p)$ are hereditarily densely s-hypercyclic sequences.
Then $s$-$HC((T_{1,n}), \dots ,(T_{p,n}))$ contains, except for \,$0$, a dense linear subspace of $X$.
\end{enumerate}
\end{theorem}

\begin{proof}
(a) By hypothesis, there is \,$i \in \{1, \dots , p\}$ \,such that \,$T_i T_j = T_j T_i$ $(j=1, \dots ,p)$.
Therefore $P(T_i) T_j = T_j P(T_i)$ for all \,$j$ \,and every polynomial \,$P$ \,with coefficients in \,$\K$.
Let \,$\mathcal P$ \,denote the set of such polynomials.
Of course, the operator \,$T_i$ \,is hypercyclic. From a result by Wengenroth \cite{wengenroth2003}, the operator
\,$P(T_i)$ \,has dense range as soon as \,$P \in \mathcal P \setminus \{0\}$. Pick any \,$x_0 \in s$-$HC(T_1, \dots , T_p)$.
Let us define \,$M := \{P(T_i)x_0: \, P \in \mathcal P \setminus \{0\} \}$. Then \,$M$ \,is a linear subspace of \,$X$.
It is dense because \,$M$ \,contains the orbit \,$\{T_i^nx_0:n\in\mathbb{N}\}$, that is dense in $X$ as \,$x_0 \in HC(T_i)$.
It remains to show that \,$M \setminus \{0\} \subset s$-$HC(T_1, \dots , T_p)$.

\vskip .15cm

To this end, fix\,$u \in M \setminus \{0\}$. Then there is \,$P \in \mathcal P \setminus \{0\}$ \,such that
\,$u = P(T_i)x_0$. It must be proved that
$$Z \supset \Delta (X^p),$$
where \,$Z := \ovl{\{(T_1^n u, \dots , T_p^n u):n\in\mathbb{N}\}} = \ovl{\{(P(T_i)T_1^n x_0, \dots , P(T_i) T_p^n x_0):n\in\mathbb{N}\}}$,
where the last equality follows from commutativity. We know that \,$\Delta (X^p) \subset $
$\ovl{\{(T_1^n x_0, \dots , T_p^n x_0):n\in\mathbb{N}\}}$.
Let \,$A := \{(T_1^n x_0, \dots , T_p^n x_0):n\in\mathbb{N}\}$, $\varphi := P(T_i)$ \,and \,$\Phi:X^p \to X^p$ \,be the mapping defined
as \,$\Phi (x_1, \dots , x_p) := (\varphi (x_1), \dots , \varphi (x_p))$. Then, as \,$\varphi$ \,is continuous, we get
$$Z = \ovl{\Phi (A)} \supset \Phi (\ovl{A}) \supset \Phi (\Delta (X^p)) = \{(\varphi (x), \dots ,\varphi (x)) :\, x \in X\},$$
so \,$Z \supset \ovl{\{(\varphi (x), \dots ,\varphi (x)) :\, x \in X\}}$.
Given \,$y \in X$ \,and a neighborhood \,$U$ \,of \,$(y,y, \dots ,y)$, there exists a neighborhood \,$V$ \,of \,$y$
\,such that \,$U \supset V^p$. Since \,$\varphi$ \,has dense range, one can find \,$x \in X$ \,with \,$\varphi (x) \in V$.
Then \,$(\varphi (x), \dots , \varphi (x)) \in U$. In other words, $(y, \dots ,y) \in \ovl{\{(\varphi (x), \dots ,\varphi (x)) :\, x \in X\}}$,
so \,$(y, \dots ,y) \in Z$. Consequently, $Z \supset \Delta (X^p)$, as required.

\vskip .15cm

\noindent (b)--(c). By mimicking the proofs of Theorems 1--2 of \cite{bernal1999b}
(in which the results are given for a single sequence $(T_n)$), we can construct recursively a
sequence \,$(x_N)_{N\in\mathbb{N}} \subset X$ \,and a family $\{(q(N,k))_{k\in\mathbb{N}}:N\in\mathbb{N}_0\}$ of strictly increasing subsequences of \,$\N$
\,satisfying, for all \,$N \in \N$, the following conditions: $x_N \in G_N \cap s\mbox{-}HC((T_{1,q(N-1,k)}, \dots ,(T_{p,q(N-1,k)}))$ \,and \,$T_{j,q(l,k)}x_N \to 0$ as $k \to \infty$ for all $l \ge N$
and all $j \in \{1, \dots ,p\}$, where \,$G_0 := X$ and $G_N := X \setminus {\rm span}\,\{x_1, \dots ,x_{N-1}\}$ $(N\in\mathbb{N})$ if the assumptions
of (b) hold, while \,$\{G_N\}_{N \in \N}$ \,denotes any fixed open basis of \,$X$ if the assumptions of (c) hold.
Then \,$M := {\rm span} \, \{x_N: \, N \in \N\}$ \,is the sought-after vector subspace. The details are left as an exercise.
\end{proof}

\section{s-Universality criteria}

\quad A number of workable sufficient conditions will be useful to detect s-universality. Recall that a sequence of continuous mappings $T_n:X \to Y$ $(n\in\mathbb{N})$ is called {\it mixing} provided that,given nonempty open sets $U \subset X$, $V \subset Y$, there is $N \in \N$ such that $T_n (U) \cap V \ne \varnothing$ for all $n \ge N$. The corresponding notion of simultaneous mixing property arises naturally, as well as the one of simultaneous transitivity. Note that $T_n (U) \cap V \ne \varnothing$ is equivalent to $U \cap T_n^{-1} (V) \ne \varnothing$.

\begin{definition}  \label{definition of s-transitive}
{\rm
Let $p \in \N$ and $X,Y$ be Hausdorff topological
spaces. Assume that, for each $j \in \{1, \dots ,p\}$, $T_{j,n}:X
\to Y$ $(n \in \N )$ is a sequence of  continuous
mappings. Let also \,$T_1, \dots ,T_p : X \to X$ \,be continuous mappings. We say that:
\begin{enumerate}
\item [\rm (a)] The sequences $(T_{1,n}), \dots , (T_{p,n})$
are {\it simultaneously transitive}
(or {\it s-transitive}) \,provided that,
for every pair of nonempty open sets $U \subset X$, $V \subset Y$, there is $N \in \N$ such that
$U \cap \bigcap_{j=1}^p T_{j,N}^{-1} (V) \ne \varnothing$.
\item [\rm (b)] The sequences $(T_{1,n}), \dots , (T_{p,n})$
are {\it simultaneously mixing}
(or {\it s-mixing}) \,provided that,
for every pair of nonempty open sets $U \subset X$, $V \subset Y$, there is $N \in \N$ such that
$U \cap \bigcap_{j=1}^p T_{j,n}^{-1} (V) \ne \varnothing$ for all $n \ge N$.
\item [\rm (c)] The mappings \,$T_1, \dots ,T_p$ \,are {\it simultaneously transitive}
({\it simultaneously mixing}, resp.) whenever the sequences
$(T_1^n), \dots ,(T_p^n)$ are s-transitive (s-mixing, resp.).
\end{enumerate}
}
\end{definition}

\begin{remark}
Corresponding concepts of d-transitivity and d-mixing were introduced in \cite{besperis2007}, where $\bigcap_{j=1}^p T_{j,n}^{-1} (V_j)$
($V_j$ nonempty open subsets of $Y$, $j = 1, \dots ,p$) appears instead of \,$\bigcap_{j=1}^p T_{j,n}^{-1} (V)$. Also, most criteria
given in this section have their counterparts for the related d-properties as provided in \cite{bernal2007} and \cite{besperis2007}.
A thorough study of d-mixing operators is provided in \cite{besmartinperisshkarin2012}.
\end{remark}

Note that, contrary to the one-sequence case, the facts $U \cap \bigcap_{j=1}^p T_{j,N}^{-1} (V) \ne \varnothing$ and
$\bigcap_{j=1}^p T_{j,N} (U) \cap V \ne \varnothing$ are not equivalent.
Observe also that $\bigcap_{j=1}^p T_{j,n}^{-1} (V) = [T_{1,n}, \dots , T_{p,n}]^{-1} (V^p)$. From the definitions, it is easy to check that
the sequences $(T_{1,n}), \dots , (T_{p,n})$ are s-mixing if and only if, for every strictly increasing sequence $(n_k)$ in $\N$,
the sequences $(T_{1,n_k}), \dots , (T_{p,n_k})$ are s-transitive. The following proposition provides what can be called the
Birkhoff s-transitivity theorem.

\begin{proposition} \label{Prop:s-transitive-s-mixing}
Under the same assumptions and terminology as in Definition \ref{definition of s-transitive}, let us suppose, in addition,
that \,$X$ is Baire and \,$Y$ is second-countable. Then we have:
\begin{enumerate}
\item [\rm (i)] The sequences $(T_{1,n}), \dots ,(T_{p,n})$ are s-transitive if and only if they are densely s-universal.
If this is the case, then the set $s$-$\mathcal U ((T_{1,n}), \dots , (T_{p,n}))$ is residual in $X$.
\item [\rm (ii)] The sequences $(T_{1,n}), \dots ,(T_{p,n})$ are s-mixing if and only if, for every strictly increasing sequence $(n_k) \subset \N$,
the sequences $(T_{1,n_k}), \dots ,(T_{p,n_k})$ are densely s-universal.
\end{enumerate}
\end{proposition}

\begin{proof}
Part (ii) is an immediate consequence of (i). Let us prove (i). Fix a coun\-ta\-ble open basis $(V_m)$ of $Y$, as well as a point $x_0 \in X$.
Then $x_0 \in s$-$\mathcal U ((T_{1,n}), \dots , (T_{p,n}))$ if and only if, given a nonempty open set $V \subset Y$, there is $n \in \N$ with
$[T_{1,n}, \dots , T_{p,n}] x_0 \in V^p$, that is, $x_0 \in \bigcup_{n \in \N} \bigcap_{j=1}^p T_{j,n}^{-1} (V)$. Since each $V$ contains some $V_m$ and each $V_m$ is a nonempty subset of $Y$, the last property is the same as  $x_0 \in \bigcap_{m \in \N} \bigcup_{n \in \N} \bigcap_{j=1}^p T_{j,n}^{-1} (V_m)$, which shows that
$$
s\hbox{-}\mathcal U ((T_{1,n}), \dots , (T_{p,n})) = \bigcap_{m \in \N} \bigcup_{n \in \N} \bigcap_{j=1}^p T_{j,n}^{-1} (V_m). \eqno (1)
$$
Since the $T_{j,n}$'s are continuous, each set $\bigcap_{j=1}^p T_{j,n}^{-1} (V_m)$ is open. If $(T_{1,n}), \dots ,(T_{p,n})$ are s-transitive
then every set $\bigcup_{n \in \N} \bigcap_{j=1}^p T_{j,n}^{-1} (V_m)$ $(m \in \N )$ is (open and) dense. Hence their (countable) intersection,
which equals $s$-$\mathcal U ((T_{1,n}), \dots , (T_{p,n}))$ by (1), is a dense $G_\delta$ subset (so residual) in $X$ because $X$ is Baire.
Conversely, assume that the set of s-universal elements is dense in $X$ and fix a nonempty open subset $V$ of $Y$. Then there is $m\in \N$ with
$V \supset V_m$. It follows from (1) that $\bigcup_{n \in \N} \bigcap_{j=1}^p T_{j,n}^{-1} (V_m)$ is dense in $X$, so the bigger set
$\bigcup_{n \in \N} \bigcap_{j=1}^p T_{j,n}^{-1} (V)$ is also dense. But this means that, given a nonempty set $U \subset X$, there is
$N \in \N$ such that $U \cap \bigcap_{j=1}^p T_{j,N}^{-1} (V) \ne \varnothing$ or, in other words,
the sequences $(T_{1,n}), \dots ,(T_{p,n})$ are s-transitive.
\end{proof}

In the linear case, we state the following set of sufficient conditions, that are inspired by the results
contained in \cite[Sect.~1c]{grosse1999} and the references cited in it.

\begin{theorem} \label{Theorem:dense-s-hc}
Let $X$ and $Y$ be topological vector spaces such that $X$ is Baire and \,$Y$ is metrizable and separable,
and let \,$(T_{j,n})_{n\in\mathbb{N}}$ $(j=1, \dots ,p )$ \,be sequences in \,$L(X,Y)$. Assume that there are
respective dense subsets \,$X_0$ of \,$X$ and \,$Y_0$ of \,$Y$ satisfying at least one of the following conditions: %_{n \ge 1}
\begin{enumerate}
\item[\rm (A)] For every pair of vectors \,$x \in X_0, \, y \in Y_0$, there exist sequences \,$(n_k) \subset \N$
               \,and \,$(x_k) \subset X$ \,with \,$x_k \to 0$, $T_{j,n_k} x \to 0$ \,and \,$T_{j,n_k} x_k \to y$
               $(j=1, \dots , p)$ \,as \,$k \to \infty$.
\item[\rm (B)] For every \,$x \in X_0$, the sequences \,$(T_{j,n}x)_{n\in\mathbb{N}}$ $(j=1, \dots ,p )$ \,converge in \,$Y$ \,to a common limit and,
               for every $y \in Y_0$, there exist sequences \,$(n_k) \subset \N$
               \,and \,$(x_k) \subset X$ \,with \,$x_k \to 0$ \,and \,$T_{j,n_k} x_k \to y$
               $(j=1, \dots , p)$ \,as \,$k \to \infty$.
\item[\rm (C)] For every \,$x \in X_0$, there exists a sequence \,$(n_k) \subset \N$ \,such that
               the sequences \,$(T_{j,n_k}x)_{k\in\mathbb{N}}$ $(j=1, \dots ,p )$ \,converge in \,$Y$ \,to a common limit and, for every $y \in Y_0$, there exists a sequence \,$(x_n) \subset X$ \,such that \,$x_n \to 0$ \,and \,$T_{j,n} x_n \to y$
               $(j=1, \dots , p)$ \,as \,$n \to \infty$.
\end{enumerate}
Then \,$(T_{j,n})_{n\in\mathbb{N}}$ $(j=1, \dots ,p )$ \,are densely s-hypercyclic. %_{n \ge 1}
\end{theorem}

\begin{proof}
According to Proposition \ref{Prop:s-transitive-s-mixing}, we should show that \,$(T_{j,n})_{n\in\mathbb{N}}$ $(j=1, \dots ,p )$
\,are s-transitive. With this aim, fix a pair of nonempty open sets \,$U \subset X$, $V \subset Y$. We should exhibit an \,$N \in \N$
\,such that \,$U \cap \bigcap_{j=1}^p T_{j,N}^{-1} (V) \ne \varnothing$.

\vskip .15cm

Assume first that (A) holds. By density, there are $x \in X_0$ and $y \in Y_0$ such that $x \in U$ and $y \in V$.
Define $A := U - x$ and $B := V - y$. Then $A$ and $B$ are open neighborhoods of \,$0$ \,in $X$ and $Y$ respectively.
Take a $0$-neighborhood $C \subset Y$ satisfying $C + C \subset B$.
Consider the sequences $(n_k)$ and $(x_k)$ provided by (A). Then there is $k \in \N$ such that $x_k \in A$,
$T_{j,n_k} x \in C$ \,and \,$T_{j,n_k} x_k \in y + C$ $(j=1, \dots ,p )$. Let $u := x + x_k$ and $N := n_k$. We get
$u \in x + A = U$ and $T_{j,N} u = T_{j,N} x + T_{j,N} x_k \in C + y + C \subset y + B = V$ $(j=1, \dots ,p )$,
so that \,$u \in U \cap \bigcap_{j=1}^p T_{j,N}^{-1} (V)$.

\vskip .15cm

Suppose now that (B) holds. By density, there is $x \in X_0$ such that $x \in U$. Define $A := U - x$, a neighborhood of $0$.
By hypothesis, there is $z \in Y$ such that $T_{j,n} \to z$ as $n \to \infty$ $(j=1, \dots ,p)$. Since $Y_0$ is dense in $Y$,
there is $y \in Y_0$ with $y \in z + V$. Let $B := V - y + z$, a neighborhood of $0$ in $Y$.
Take a $0$-neighborhood $C \subset Y$ satisfying $C + C \subset B$. We have that $T_{j,n} \in z + C$ $(j=1, \dots ,p)$
for $n \ge n_0$, say. Consider the sequences $(n_k)$ and $(x_k)$ provided by (B) for the vector $y - z$, so that
$x_k \to 0$ and  $T_{j,n_k} x_k \to y-z$ $(j=1, \dots , p)$ \,as \,$k \to \infty$. Choose $k \in \N$ so large that $n_k \ge n_0$,
$x_k \in A$ and $T_{j,n_k} x_k \in y-z + C$ $(j=1, \dots , p)$. Let $u := x + x_k$ and $N := n_k$. Then
$u \in x + A = U$ and, for every $j = 1, \dots ,p$,
$$
T_{j,N} u = T_{j,N} x + T_{j,N} x_k \in z + C + y - z + C = y + C + C \subset y + B = V,
$$
so that \,$u \in U \cap \bigcap_{j=1}^p T_{j,N}^{-1} (V)$, as required. Under assumption (C), the proof is similar and left as an exercise.
\end{proof}

Two of the most popular criteria of hypercyclicity are the so-called blow-up/co\-l\-lap\-se criterion and the hypercyclicity criterion
(see \cite{bayartmatheron2009,grosse2003,grosseperis2011}). Now, we can obtain their respective s-versions.

\begin{proposition}
{\rm [s-Blow-up/Collapse Criterion]} Let $X$ be a Baire metrizable separable topological vector space,
and let \,$(T_{j,n})_{n\in\mathbb{N}}$ $(j=1, \dots ,p )$ \,be sequences in \,$L(X)$. %_{n \ge 1}
Suppose that, for every nonempty open subsets $U,V$ of $X$ and every $0$-neighborhood \,$W \subset X$ there is $N \in \N$ such that
$$W \cap \bigcap_{j=1}^p T^{-1}_{j,N} (V) \ne \varnothing \ne U \cap \bigcap_{j=1}^p T^{-1}_{j,N} (W).$$
Then \,$(T_{j,n})_{n\in\mathbb{N}}$ $(j=1, \dots ,p )$ \,are densely s-hypercyclic.
\end{proposition}

\begin{proof}
Fix a pair of nonempty open sets $U,V \subset X$. Choose vectors $x \in U$, $y \in V$.
It suffices to exhibit sequences
sequences \,$(n_k) \subset \N$ \,and \,$(x_k) \subset X$ \,with \,$x_k \to x$ \,and \,$T_{j,n_k} x_k \to y$
$(j=1, \dots , p)$, because this would entail the existence of some $k \in \N$ such that $x_k \in U$ \,and \,$T_{j,n_k} x_k \in V$
$(j=1, \dots , p)$, so $x_k \in U \cap \bigcap_{j=1}^p T_{j,n_k}^{-1} (V)$. In other words, the sequences $(T_{j,n})_{n\in\mathbb{N}}$ $(j=1, \dots ,p )$
\,would be s-transitive, hence densely s-hypercyclic by Proposition \ref{Prop:s-transitive-s-mixing}.

\vskip .10cm

With this aim, choose a fundamental decreasing sequence $(W_k)$ of $0$-neigh\-bor\-hoods. Then $(U_k) := (x + W_k)$ and
$(V_k) := (y + W_k)$ are fundamental decreasing sequences of $x$-neighborhoods and $y$-neighborhoods, respectively.
By hypothesis, for each $k \in \N$, there are $n_k \in \N$ and points $x_k'$ and $x_k''$ such that
$x_k' \in W_k \cap \bigcap_{j=1}^p T^{-1}_{j,n_k} (V_k)$ and $x_k'' \in U_k \cap \bigcap_{j=1}^p T^{-1}_{j,n_k} (W_k)$.
Let $x_k := x_k' + x_k''$. Then $x_k \to x$ as $k \to \infty$ because $x_k' \in W_k$ (so $x_k' \to 0$) and $x_k'' \in U_k$ (so $x_k'' \to x$). Finally, $T_{j,n_k} x_k = T_{j,n_k} x_k' + T_{j,n_k} x_k'' \to y + 0 = y$ $(j=1, \dots , p)$ because $T_{j,n_k} x_k' \in V_k$ and $T_{j,n_k} x_k'' \in W_k$ for all $k \in \N$.
\end{proof}

%\begin{corollary}
%Let $X$ and $Y$ be topological vector spaces such that $X$ is Baire and \,$Y$ is metrizable and separable,
%and let \,$(T_{j,n})_{n \ge 1}$ $(j=1, \dots ,p )$ \,be sequences in \,$L(X,Y)$. Assume that there are
%dense subsets \,$X_0$ of \,$X$ and \,$Y_0$ of \,$Y$, as well as a sequence \,$(n_k) \subset \N$ \,satisfying the following property:
%for every pair of vectors \,$x \in X_0, \, y \in Y_0$ \,there exists a sequence
%\,$(x_k) \subset X$ \,with
%$$
%x_k \to 0, \,\,\, T_{j,n_k} x \to 0 \hbox{ \ and \ } \,T_{j,n_k} x_k \to y \,\,\, (j=1, \dots , p) \hbox{ \,as } \, k \to \infty .
%$$
%Then \,$(T_{j,n})_{n \ge 1}$ $(j=1, \dots ,p )$ \,are densely s-hypercyclic.
%\end{corollary}
%
%\begin{remark} \label{remark s-hc c}
%The property in the last corollary can be formulated equivalently as: There are mappings $S_k:Y_0 \to X$ $(k \ge 1)$ such that
%$T_{j,n_k} \to 0$ $(j=1, \dots , p)$ pointwise on $X_0$, $S_k \to 0$ \,pointwise on \,$Y_0$ and \,$T_{j,n_k}S_k \to Id$ $(j=1, \dots , p)$ \,pointwise on $Y_0$.
%This form fits better the common formulation of the hypercyclic criterion in the related literature (see \cite{bayartmatheron2009,grosseperis2011}).
%\end{remark}

%The assumptions of the s-hypercyclicity criterion might be somewhat restrictive
%in order to be applied to not too similar operators. This is why we are going to introduce a variant of such criterion.

\vspace{6pt}
Recall that the convex hull \,conv$(A)$ \,of a subset $A$ of a vector space $X$ is the least convex subset of $X$ containing $A$.

\begin{definition} \label{definition of s-hc Delta criterion}
Let $X$ be a Baire metrizable separable locally convex space, $(n_k) \subset \N$ be a
strictly increasing sequence and \,$T_j \in L(X)$ $(j=1, \dots ,p)$.
We say that $T_1, \dots ,T_p$ {\it satisfy the s-hypercyclicity criterion} with respect
to $(n_k)$ if there are subsets $X_0 \subset X, \, W_0 \subset X^p$ such that $X_0$ is dense in $X$ and
    $$
    \ovl{W_0} \supset \Delta (X^p)
    $$
as well as mappings \,$R_k : W_0 \to X$ $(k \in \N )$ \,such that
\begin{enumerate}
\item[\rm (i)] $T_j^{n_k} \to 0$ \,pointwise on $X_0$ as $k \to \infty$ $(j=1, \dots , p)$,
\item[\rm (ii)] $R_k \to 0$ \,pointwise on $W_0$ as $k \to \infty$ and
\item[\rm (iii)] For every $w = (w_1, \dots ,w_p) \in W_0$ and every $j \in \{1, \dots ,p\}$ there is $y_j \in {\rm conv} (\{w_1, \dots ,w_p\})$
such that $T_j^{n_k} R_k w \to y_j$ as $k \to \infty$.
\end{enumerate}
\end{definition}

\begin{theorem} \label{Prop: s-hc Delta-criterion}
{\rm [s-Hypercyclicity Criterion]} Let $X$ be a Baire metrizable se\-pa\-ra\-ble locally convex space and \,$T_j \in L(X)$
$(j=1, \dots ,p)$. If \,$T_1, \dots ,T_p$ sa\-tis\-fy the s-hypercyclicity criterion with respect to some $(n_k) \subset \N$, then \,$(T_1^{n_k}), \dots , (T_p^{n_k})$ are s-mixing. In particular, $T_1, \dots ,T_p$ are densely s-hypercyclic.
\end{theorem}

\begin{proof}
Let $U,V \subset X$ be nonempty open sets. Then there are $x_0 \in U \cap X_0$ and $y_0 \in V$.
By local convexity, there is a convex open set $\widetilde{V}$ with $y_0 \in \widetilde{V} \subset V$.
As $(y_0, \dots ,y_0) \in \Delta (X^p) \subset \ovl{W_0}$, one can find $w = (w_1, \cdots ,w_p) \in W_0$ such that $w_j \in \widetilde{V}$
for all $j=1, \dots ,p$. Put
$$z_k := x_0 + R_k w \quad (k \in \N ).$$
Then, due to (ii), $z_k \to x_0 + 0 = x_0 \in U$ as $k \to \infty$. Moreover, for every $j \in \{1, \dots ,p\}$ we get thanks to (i) and (iii) that
$$T_j^{n_k} z_k = T_j^{n_k} x_0 + T_j^{n_k} R_k w \longrightarrow 0 + y_j = y_j \hbox{ \ as \ } k \to \infty ,$$
where, for each $j$, $y_j \in {\rm conv} (\{w_1, \dots ,w_p\}) \subset {\rm conv} (\widetilde{V}) = \widetilde{V} \subset V$.
Consequently, there is $k_0 \in \N$ such that, for all $k \ge k_0$, we have $z_k \in U$ and $T_j^{n_k} z_k \in V$ $(j=1, \dots ,p)$ or, in other words,
$z_k \in U \cap \bigcap_{j=1}^p T_j^{-n_k} (V) \ne \varnothing$, as required.
\end{proof}

\begin{remarks} \label{Remarks Delta + disjoint}
1. Examples of spaces $X$ satisfying the assumptions of Theorem \ref{Prop: s-hc Delta-criterion} are the Fr\'echet spaces, that is, the locally convex F-spaces. If local convexity is dropped from the assumptions, then the conclusion still holds if we replace (iii) by the (stronger) condition:

\vskip 3pt

\noindent (iii') {\it $T_j^{n_k} R_k w \to w_j$ as $k \to \infty$ for every $w = (w_1, \dots ,w_p) \in W_0$ and every \hfil\break
\phantom{aaaa} $j \in \{1, \dots ,p\}$.}

\vskip 5pt

\noindent 2. In \cite[Proposition 2.6]{besperis2007} the following {\it d-hypercyclicity criterion} was proved,
where $X$ is a Fr\'echet space, $(n_k) \subset \N$ is a
strictly increasing sequence and \,$T_j \in L(X)$ $(j=1, \dots ,p)$:
Assume that there exist dense subsets $X_0,X_1, \dots ,X_p \subset X$
and mappings $S_{k,j}: X_j \to X$ $(k\in\mathbb{N}; \,1 \le j \le p)$ satisfying
$T_j^{n_k} \to 0$ $(k \to \infty )$ \,pointwise on $X_0$,
$S_{k,j} \to 0$ $(k \to \infty )$ \,pointwise on $X_j$, and $T_j^{n_k}S_{k,l} \to \delta_{j,l} \id_{X_l}$ $(k \to \infty )$ pointwise on $X_l$ $(1 \le j,l \le p)$.
Then \,$(T_1^{n_k}), \dots , (T_p^{n_k})$ are d-mixing (see \cite[Definition 2.1]{besperis2007}). In particular, by \cite[Proposition 2.3]{besperis2007}, $T_1, \dots ,T_p$ are densely d-hypercyclic.

Now, we can obtain a disjoint hypercyclicity criterion under weaker assumptions. Namely, let us assume that
there are dense subsets $X_0 \subset X, \, W_0 \subset X^p$ and mappings $R_k : W_0 \to X$ $(k \in \N )$ satisfying (i)--(ii) of
Definition \ref{definition of s-hc Delta criterion} together with (iii') of the preceding remark (it is easy to check that these assumptions are weaker than those of the d-hypercyclicity criterion in \cite{besperis2007}).
Then $(T_1^{n_k}), \dots , (T_p^{n_k})$ are d-mixing. Indeed,
let $U,V_1, \dots ,V_p \subset X$ be nonempty open sets. By density, there are $x_0 \in U \cap X_0$ and
$w = (w_1, \cdots ,w_p) \in W_0 \cap (V_1 \times \cdots \times V_p)$. Let
$z_k := x_0 + R_k w$ $(k \in \N )$.
Then $z_k \to x_0 + 0 = x_0 \in U$ as $k \to \infty$. Moreover, for every $j \in \{1, \dots ,p\}$ we get
$T_j^{n_k} z_k = T_j^{n_k} x_0 + T_j^{n_k} R_k w \longrightarrow 0 + w_j = w_j$ as $k \to \infty$.
Consequently, there is $k_0 \in \N$ such that, for all $k \ge k_0$, we have $z_k \in U$ and $T_j^{n_k} z_k \in V_j$ $(j=1, \dots ,p)$,
that is,
$z_k \in U \cap \bigcap_{j=1}^p T_j^{-n_k} (V_j) \ne \varnothing$, which is the d-mixing property.

\vskip 5pt

\noindent 3. Several sets of conditions on $T_1, \dots , T_p$ such that these operators satisfy the s-hypercyclicity criterion
with respect to a strictly increasing sequence $(n_k) \subset \N$ are --as it is easy to check-- the following:
\begin{enumerate}
\item[\rm (a)] There are dense subsets $X_0,Y_0\subset X$ and mappings
$S_{k,j}:Y_0 \to X$ $(k\in\mathbb{N}; \,1 \le j \le p)$ such that (i) holds, $\sum_{j=1}^pS_{k,j}\to0$ pointwise on $Y_0$
and $T_j^{n_k} \sum_{l=1}^p S_{k,l} \to \id_{Y_0}$ pointwise on $Y_0$ $(j=1, \dots , p)$.
\item[\rm (b)] There are subsets $X_0,X_1, \dots ,X_p\subset X$ in such a way that $X_0$ is dense in $X$ and $\ovl{X_1 \times \cdots \times X_p} \supset \Delta (X^p)$ as well as mappings $S_{k,j}:X_j \to X$ $(k\in\mathbb{N}; \,1 \le j \le p)$ such that
(i) holds, $\sum_{j=1}^p S_{k,j} x_j \to 0$ for all $(x_1, \dots ,x_p) \in X_1 \times \cdots \times X_p$, and
$T_j^{n_k} (\sum_{l=1}^p S_{k,l} x_l) \to x_j$ for all
$(x_1, \dots ,x_p) \in X_1 \times \cdots \times X_p$ and all $j = 1, \dots ,p$.
\end{enumerate}
In view of (b), we see that if $T_1,\dots,T_p$ satisfy the d-hypercyclicity criterion with respect to $(n_k)$, then they also satisfy the s-hypercyclicity criterion with respect to $(n_k)$.
\end{remarks}

B\`es and Peris \cite{besperis1999} have proved that satisfaction of the hypercyclicity criterion, hereditary
hypercyclicity and transitivity of self-sums are equivalent (see also \cite{bernalgrosse2003}).
Moreover, they es\-ta\-bli\-shed a similar result for d-hypercyclicity \cite[Theorem 2.7]{besperis2007}.
Now, we prove that a corresponding statement also holds for s-hypercyclicity, with the d-hypercyclicity criterion replaced
by the s-hypercyclicity criterion (Theorem \ref{Prop: s-hc Delta-criterion}), so showing that the latter is rather natural.

\begin{proposition}
Let $X$ be a separable Fr\'echet space and $T_j \in L(X)$ \break
$(j=1, \dots ,p)$. Consider the following statements:
\begin{enumerate}
\item[\rm (a)] $T_1, \dots , T_p$ satisfy the s-hypercyclicity criterion.
\item[\rm (b)] $(T_1^{n_k}), \dots ,(T^{n_k}_p)$ are hereditarily densely s-hypercyclic for some $(n_k) \subset \N$.
\item[\rm (c)] $\oplus_{k=1}^m T_1, \dots \oplus_{k=1}^m T_p$ are s-transitive on $X^m$ for all $m \in \N$.
\item[\rm (d)] $T_1 \oplus T_1, \dots T_p \oplus T_p$ are s-transitive on $X^2$.
\end{enumerate}
Then we have:
\begin{enumerate}
\item[\rm (A)] {\rm (a), (b) and (c)} are equivalent.
\item[\rm (B)] If there exists $i \in \{1, \dots ,p\}$ such that $T_i T_j = T_j T_i$ for all $j \in \{1, \dots ,p\}$, then {\rm (a), (b), (c)}
are equivalent to {\rm (d).}
\end{enumerate}
\end{proposition}

\begin{proof}
In the proof of (A), we follow closely the proof of Theorem 2.7 in \cite{besperis2007}, while the proof of (B) runs similar as the proof of Theorem 2.3, $(3)\Rightarrow(1)$, in \cite{besperis1999}.

\vskip 5pt

\noindent (A)  (a) $\Rightarrow$ (b): $T_1, \dots , T_p$ satisfy the s-hypercyclicity criterion with respect to some $(n_k) \subset \N$, so that they also satisfy it for any subsequence $(m_k)$ of $(n_k)$. By Theorem \ref{Prop: s-hc Delta-criterion}, $(T_1^{m_k}), \dots ,(T^{m_k}_p)$ are s-mixing and therefore densely s-hypercyclic.

\vskip 5pt

\noindent (b) $\Rightarrow$ (c): Let $m \in \N$ be fixed and let $\varnothing \ne U_l,V_l \subset X$ open
$(l=1, \dots ,m)$. It suffices to show that there exists $N \in \N$ such that
$$
U_l \,\cap \,\bigcap_{j=1}^p T_j^{-N} (V_l) \ne \varnothing \hbox{ \,for all \,} l = 1, \dots ,m. \eqno (1)
$$
Since $(T_1^{m_k}), \dots ,(T^{m_k}_p)$ are densely s-hypercyclic for each subsequence $(m_k)$ of $(n_k)$,
the sequences $(T_1^{n_k}), \dots ,(T^{n_k}_p)$ are s-mixing (cf.~Proposition \ref{Prop:s-transitive-s-mixing}(ii)).
Hence, for each $l \in \{1, \dots ,m\}$, there exists \,$k_0(l) \in \N$ \,such that \,$U_l \,\cap \,\bigcap_{j=1}^p T_j^{-n_k} (V_l) \ne \varnothing$
\,for all \,$k \ge k_0(l)$. Then (1) is satisfied simply by choosing $N:= \max \{k_0(1), \dots ,k_0(m)\}$.

\vskip 5pt

\noindent (c) $\Rightarrow$ (a): Due to the assumption, we have:

\vskip 5pt

$(*)$ For every $m \in \N$ and every $2m$-tuple \,$U_1,\dots ,U_m,V_1, \dots ,V_m$ \,of nonempty \break
\phantom{aaaaa} open subsets of $X$ there is $N \in \N$ arbitrarily large such that (1) holds.

\vskip 4pt

Let $(A_n)_{n \in \N}$, $(B_n)_{n \in \N}$ be bases of nonempty sets of the topology of $X$.
For $n \in \N$, we write $W_n := B(0,1/n)$ (open $d$-balls, $d$ being a translation-invariant distance generating the topology of $X$)
and $A_{n,0} := A_n, \, B_{n,0} := B_n$.

\vskip 5pt

Choose a nonempty open set $A_{1,1}$ with ${\rm diam} (A_{1,1}) < 1/2$ and
$\ovl{A_{1,1}} \subset A_1$. Due to $(*)$ (with $m=2$), there is $n_1 \in \N$ such that $B_1 \cap \bigcap_{j=1}^pT_j^{-n_1} (W_1) \ne \varnothing$
and $W_1 \cap \bigcap_{j=1}^pT_j^{-n_1} (A_{1,1}) \ne \varnothing$. Thus, there exist a nonempty open set $B_{1,1}$
with ${\rm diam} (B_{1,1}) < 1/2$, $\ovl{B_{1,1}} \subset B_1$ and $T_j^{n_1} (\ovl{B_{1,1}}) \subset W_1$ for all $j = 1, \dots ,p$, as well as a point
$w_{1,1} \in W_1$ with $T_j^{n_1} w_{1,1} \in A_{1,1}$ for all $j = 1, \dots ,p$. Now, for $i=1,2$, choose $A_{i,3-i}$ open, nonempty, such that
diam$(A_{i,3-i}) < 1/3$, $\ovl{A_{i,3-i}} \subset A_{i,2-i}$ and $\ovl{A_{1,2}} \cap \ovl{A_{2,1}} = \varnothing$.
Due to $(*)$ (with $m=4$), there is $n_2 \in \N$ with $n_2 > n_1$ such that
$B_{i,2-i} \cap \bigcap_{j=1}^pT_j^{-n_2} (W_2) \ne \varnothing$ and
$W_2 \cap \bigcap_{j=1}^pT_j^{-n_2} (A_{i,3-i}) \ne \varnothing$ for $i = 1,2$.
Thus, there exist nonempty open sets $B_{i,3-i}$ with ${\rm diam} (B_{i,3-i}) < 1/3$, $\ovl{B_{i,3-i}} \subset B_{i,2-i}$ and
$T_j^{n_2}(\ovl{B_{i,3-i}}) \subset W_2$ $(i=1,2)$ for all $j = 1, \dots ,p$ as well as points $w_{i,3-i} \in W_2$ $(i=1,2)$ with
$T_j^{n_2} w_{i,3-i} \in A_{i,3-i}$  $(i=1,2)$ for all $j = 1, \dots ,p$.

\vskip 5pt

Continuing this process inductively, by using $(*)$ with $m=2k$ in step $k$, we obtain a strictly increasing sequence $(n_k) \subset \N$,
nonempty open sets $A_{i,k+1-i}, \, B_{i,k+1-i}$ with ${\rm diam} (A_{i,k+1-i}) < {1 \over k+1}, \, {\rm diam} (B_{i,k+1-i})  < {1 \over k+1}$
and points $w_{i,k+1-i} \in W_k$ \break
$(1 \le i \le k; \, k \in \N)$ such that
\begin{enumerate}
\item[(i)] $\ovl{A_{i,k+1-i}} \subset A_{i,k-i}$, $\ovl{B_{i,k+1-i}} \subset B_{i,k-i}$ \,for all \,$1 \le i \le k, \,\, k \in \N$.
\item[(ii)] For each $k \in \N$, the sets  $\ovl{A_{i,k+1-i}}$, $1 \le i \le k$, are pairwise disjoint.
\item[(iii)] $T_j^{n_k} (\ovl{B_{i,k+1-i}}) \subset W_k$ $(k \in \N ; 1 \le i \le k ; \, 1 \le j \le p)$, and
\item[(iv)] $T_j^{n_k} w_{i,k+1-i} \in A_{i,k+1-i}$ $(k \in \N ; \, 1 \le i \le k; \, 1 \le j \le p)$.
\end{enumerate}
For each fixed $i \in \N$, the sequences of closed sets $(\ovl{A_{i,r}})_{r \in \N}$ and $(\ovl{B_{i,r}})_{r \in \N}$ are decreasing (due to (i)) with
${\rm diam} (\ovl{A_{i,r}}), {\rm diam} (\ovl{B_{i,r}}) < {1 \over r+i}$. The completeness of $X$ implies the existence of points $a_i,b_i \in X$ $(i \in \N )$
such that $\bigcap_{r \in \N} \ovl{A_{i,r}} = \{a_i\}$ and $\bigcap_{r \in \N} \ovl{B_{i,r}} = \{b_i\}$.

\vskip 5pt

Put $X_0 := \{b_i : \, i \in \N\} \subset X$ and $W_0 := \{a_i: \, i \in \N\}^p \subset X^p$.
As $a_i \in \ovl{A_{i,1}} \subset A_{i,0} = A_i$ and $b_i \in \ovl{B_{i,1}} \subset B_{i,0} = B_i$ for all $i \in \N$, we obtain that
$X_0$ is dense in $X$ and $W_0$ is dense in $X^p$. Due to (ii), we have that $a_i \ne a_k$ whenever $i \ne k$
(indeed, if $i < k$, say, then $a_i \in \ovl{A_{i,k+1-i}}$ and $a_k \in \ovl{A_{k,1}}$, but $\ovl{A_{i,k+1-i}} \cap \ovl{A_{k,1}} = \varnothing$).
Hence, for each $k \in \N$, the function \,$R_k : W_0 \to X$ \,given by
\begin{equation*}
	R_k (a_{i_1}, \dots ,a_{i_p}) =
	\left\{
	\begin{array}{lll}
		\dis{{1 \over p} \sum_{l=1}^p w_{i_l,k+1-i_l}}& \text{if} \,\, k \ge \dis{\max_{l=1, \dots ,p} i_l} \\
		0 & \text{otherwise}
	\end{array}
	\right.
\end{equation*}
is well defined. Altogether, we have:
\begin{itemize}
\item For all $j = 1, \dots ,p$, all $i \in \N$ and all $k \ge i$ one has, due to (iii), that $T^{n_k}_j b_i \in T^{n_k}_j (\ovl{B_{i,k+1-i}}) \subset W_k = B(0,1/k)$,
so $T_j^{n_k} \to 0$ $(k \to \infty )$ pointwise on $X_0$ for every $j = 1, \dots ,p$.
\item For every $(a_{i_1}, \dots ,a_{i_p}) \in W_0$ and every $k \ge \max_{l=1, \dots ,p} i_l$, one has \break
$R_k (a_{i_1}, \dots ,a_{i_p}) = {1 \over p} \sum_{l=1}^p w_{i_l,k+1-i_l} \to 0$ $(k \to \infty)$, because $w_{i_l,k+1-i_l} \in W_k = B(0,1/k)$.
Therefore $R_k \to 0$ $(k \to \infty )$ pointwise on $W_0$.
\item For all $j = 1, \dots ,p$, all  $(a_{i_1}, \dots ,a_{i_p}) \in W_0$ and all $k \ge \max_{l=1, \dots ,p} i_l$, we get
$T_j^{n_k} R_k (a_{i_1}, \dots ,a_{i_p}) = {1 \over p} \sum_{l=1}^p T^{n_k}_j w_{i_l,k+1-i_l}$. Since
$T^{n_k}_j w_{i_l,k+1-i_l} \in A_{i_l,k+1-i_l}$ and the sequence of sets $A_{i_l,k+1-i_l}$ $(k \in \N )$ collapses to the singleton $\{a_{i_l}\}$ as $k \to \infty$ for each $l$,
we get $T_j^{n_k} R_k (a_{i_1}, \dots ,a_{i_p}) \to {1 \over p} \sum_{l=1}^p a_{i_l} \in {\rm conv} (\{a_{i_1}, \dots ,a_{i_p} \})$ as $k \to \infty$.
\end{itemize}
Thus, $T_1, \dots ,T_p$ satisfy the s-hypercyclicity criterion with respect to $(n_k)$. The proof of (A) is finished.

\vskip 6pt

\noindent (B) Obviously, (c) always implies (d). Assume now that (d) holds and that some $T_i$ commutes with all $T_j$'s.
Our goal is to prove that (a) is satisfied.

\vskip 3pt

Let us fix any vector \,$(x_0,y_0) \in$ $s$-$HC(T_1 \oplus T_1, \dots, T_p \oplus T_p)$. We claim that, for each $m \in \N$, the vector
$(x_0,T_i^m y_0)$ is also s-hypercyclic for $T_1 \oplus T_1, \dots, T_p \oplus T_p$. Indeed, as $T_i$ is hypercyclic, it has dense range,
from which one obtains, inductively, that every set $T^m_i(X)$ is dense in $X$. Put $A := X \times T_i^m (X)$, so that $A$ is dense in $X^2$.
Given $(u,v) \in A$ there is $w \in X$ such that $v = T^m_i w$. By s-hypercyclicity, there exists $(n_k) \subset \N$ such that
$T^{n_k}_j x_0 \to u$ and $T^{n_k}_j y_0 \to w$ $(k \to \infty )$ for all $j =1, \dots ,p$. Hence, for all $j$, $T^{n_k}_j x_0 \to u$ and,
by commutativity together with continuity of $T_i^m$, we get \,$T^{n_k}_j (T_i^m y_0) = T_i^m (T^{n_k}_j y_0) \longrightarrow T_i^m w = v$ $(k \to \infty )$.
Therefore $\Sigma := \ovl{\{[(T_1 \oplus T_1)^n, \dots, (T_p \oplus T_p)^n](x_0,T_i^m y_0): \, n \in \N\}} \supset \Delta (A^p)$. Since $\ovl{\Delta (A^p)}
\supset \Delta ((X^2)^p)$ and \,$\Sigma$ \,is closed, we get \,$\Sigma \supset \Delta ((X^2)^p)$, which proves the claim.

\vskip 5pt

In particular, as $y_0$ is hypercyclic for $T_i$, for each nonempty open set $U \subset X$ there exists some $u \in U$ such that
$(x_0,u)$ is s-hypercyclic for \,$T_1 \oplus T_1, \dots, T_p \oplus T_p$. Thus, fixing a decreasing basis $(U_k)$ of neighborhoods of \,$0$
\,and using induction, we can find for each $k \in \N$ some $u_k \in U_k$ and $n_k \in \N$ with $n_k > n_{k-1}$ (where $n_0 := 0$) such that
\begin{enumerate}
\item[($\alpha $)] $T^{n_k}_j x_0 \in U_k$ for all $j = 1, \dots ,p$ and
\item[($\beta $)] $T^{n_k}_j u_k \in x_0 + U_k$ for all $j = 1, \dots ,p$.
\end{enumerate}
We define \,$X_0 := \{T_i^n x_0: \, n \in \N\}$ \,and \,$W_0 := X_0^p$. Note that $X_0$ is dense in $X$ as $x_0$ is $T_i$-hypercyclic, so
$W_0$ is dense in $X^p$ (hence $\ovl{W_0} \supset \Delta (X^p)$). Now, observe that no orbit of any hypercyclic vector can be finite, that is,
$T_i^m x_0 \ne T^n_i x_0$ if $m \ne n$. Thus, for each \,$k \in \N$, the mapping
$$
R_k : (T_i^{m_1} x_0, \dots ,T_i^{m_p} x_0) \in W_0 \longmapsto {1 \over p} \cdot \sum_{l=1}^p T_i^{m_l} u_k \in X  \eqno (1)
$$
is well defined. We have:
\vspace{2pt}
\begin{enumerate}
\item[(i)] For every $j = 1, \dots ,p$ and every $m \in \N$,  $T_j^{n_k} (T_i^m x_0) = T_i^m (T_j^{n_k} x_0) \to T_i^m 0 = 0$ $(k \to \infty)$,
  where commutativity and continuity of \,$T_i$ \,together with property ($\alpha $) have been used. This shows that $T_j^{n_k} \to 0$ pointwise on $X_0$
   for all $j=1, \dots ,p$.
\item[(ii)] From the continuity of each $T_i^{m_l}$ and the fact that $u_k \in U_k$ (hence $u_k \to 0$), it follows that $T_i^{m_l} u_k \to 0$ $(k \to \infty )$ for every $l=1, \dots ,p$. Then one derives from (1) that $R_k \to 0$ pointwise on $W_0$.
\item[(iii)] For every $j = 1, \dots ,p$ and every $(m_1, \dots ,m_p) \in \N^p$, it follows from (1) that
\begin{equation*}
\begin{split}
T_j^{n_k} R_k (T_i^{m_1} x_0, \dots ,T_i^{m_p} x_0) &= {1 \over p} \cdot \sum_{l=1}^p T_j^{n_k}T_i^{m_l} u_k 
={1 \over p} \cdot \sum_{l=1}^p T_i^{m_l} T_j^{n_k} u_k \\ \longrightarrow\, {1 \over p} \cdot &\sum_{l=1}^p T_i^{m_l} x_0 \,\in\, {\rm conv}(\{T_i^{m_1} x_0, \dots ,T_i^{m_p} x_0\})
\end{split}
\end{equation*}
as $k \to \infty$,
because of $(\beta )$ (which implies $T^{n_k}_j u_k \to x_0$) together with the commutativity and continuity of each $T_i^{m_l}$.
%Now, ${1 \over p} \cdot \sum_{l=1}^p T_i^{m_l}x_0 \in {\rm conv}(\{T_i^{m_1} x_0, \dots ,T_i^{m_p} x_0\})$.
%To summarize, for every $w \in W_0$, $(T_j^{n_k} R_k w)_{k \ge 1}$ converges to some vector in
%\,${\rm conv}(\{T_i^{m_1} x_0, \dots ,T_i^{m_p} x_0\})$.
\end{enumerate}
\vspace{2pt}
This tells us that \,$T_1, \dots , T_p$ \,satisfy the s-hypercyclicity criterion, as required.
\end{proof}

\vspace{4pt}
We raise here the question whether (d) is equivalent to (a)-(b)-(c) without assuming any commutativity.

\section{Scalar multiples of an operator}

\quad We start by studying s-hypercyclicity of scalar multiples of one operator.
We have already pointed out that there is no chance of d-hypercyclicity in this case.

\vskip .15cm

Recall that an operator $T$ on a topological vector space $X$ is called {\it hereditarily hypercyclic} whenever
 $(T^{n_k})$ is universal for every strictly increasing sequence $(n_k) \subset \N$. It is well known --and easy to see-- that, if $X$ is an F-space and
 $T \in L(X)$, then $T$ is hereditarily hypercyclic if and only if $T$ is mixing.

\begin{proposition} \label{Prop: multiples}
Let \,$X$ be a topological vector space, $p \in \N$,
$c_1, \dots ,c_p \in \K$ \,and \,$T, T_1, \dots , T_p \in L(X)$. We have:
\begin{enumerate}
\item[\rm (a)] Assume that \,$X$ is metrizable and locally convex. If \,$T, c_1 T, \dots , c_p T$ are s-hypercyclic then the $c_j$'s are unimodular, that is, $|c_1| = \cdots = |c_p| = 1$.
\item[\rm (b)] Suppose that \,$X$ is metrizable. If \,$T \in L(X)$ is hereditarily hypercyclic and the scalars $c_j$ are unimodular, then \,$T, c_1 T, \dots , c_p T$ are densely s-hypercyclic.
\end{enumerate}
\end{proposition}

\begin{proof}
(a) Assume that \,$T, c_1 T, \dots , c_p T$ are s-hypercyclic, and fix $j \in \{1, \dots ,p\}$. Let \,$c := c_j$.
Then \,$T, cT$ are s-hypercyclic, so there is \,$x_0 \in s$-$HC(T,cT)$. Since $X$ is metrizable, we can find a sequence $(n_k) \subset \N$
such that \,$T^{n_k} x_0 \to x_0$ \break
and \,$c^{n_k} T^{n_k} x_0 \to x_0$ \,as $k \to \infty$.
Of course, $x_0 \ne 0$. But $X$ is locally convex, so its topology is defined by
a separating family of seminorms. Therefore there is a continuous seminorm \,$q$ \,on \,$X$ \,such that \,$q(x_0) > 0$.
Consider the sequence of vectors
$$
u_k := (c^{n_k} - 1)T^{n_k}x_0 \quad (k\in\mathbb{N}).
$$
On the one hand, we have $u_k =
c^{n_k}T^{n_k}x_0 - T^{n_k}x_0 \to x_0 - x_0 = 0$, so \,$q(u_k) \to 0$ \,by the continuity of \,$q$. On the other hand, we get
\,$q(u_k) = |c^{n_k} - 1| q(T^{n_k}x_0)$, hence \,$|c^{n_k} - 1| = {q(u_k) \over q(T^{n_k}x_0)} \to {0 \over q(x_0)} = 0$. Therefore
$c^{n_k} \to 1$ as $k \to \infty$, which implies \,$|c| = 1$, that is, $|c_j| = 1$, as required.

\vskip .15cm

\noindent (b) The result is trivial if \,$\K = \R$ (for $c_j = \pm 1$, so $(c_jT)^{2n} = T^{2n}$ for all $n$ and all $j= 1, \dots ,p$).
The complex case \,$\K = \C$ \,is more delicate. Recall that a subset \,$E \subset \T := \{|z|=1\}$ \,is said to be a \textit{Dirichlet set}
provided that there is a strictly increasing sequence $(n_k) \subset \N$ such that \,$\sup_{z \in E} |z^{n_k} - 1| \to 0$ as $k \to \infty$.
It is well-known that every finite subset of \,$\T$ \,is Dirichlet (see \cite[Theorem 8.138(a)]{bukovsky2011}). In particular, there exists
\,$(n_k) \subset \N$ \,strictly increasing such that \,$c_j^{n_k} \to 1$ $(j=1, \dots ,p)$.
According to the hypothesis, we may take \,$x_0 \in HC((T^{n_k}))$.
Given \,$x \in X$, there is a subsequence \,$(m_k) \subset (n_k)$ \,such that \,$T^{m_k}x_0 \to x$ \,and, of course,
\,$c_j^{m_k} \to 1$ $(j=1, \dots ,p)$ \,as \,$k \to \infty$. Therefore, we obtain $(c_jT)^{m_k} x_0 \to x$ for all $j=1,\dots,p$ and hence $HC((T^{n_k})) \subset s$-$HC(T,c_1T, \dots ,c_pT)$. But \,$HC((T^{n_k}))$ is dense, so \,$s$-$HC(T,c_1T, \dots ,c_pT)$ \,also is.
\end{proof}

\begin{remarks}
\noindent 1. In part (b) of the last proposition, hereditary hypercyclicity is needed in order to obtain common subsequences
$(n_k)$ to perform approximations. If this is not claimed, then, by a result due to Le\'on and M\"uller,
any unimodular multiple of a hypercyclic operator on any topological vector
space is always hypercyclic, even with the same set of hypercyclic vectors (see \cite{leonmuller2004}
and \cite[pp.~339--340]{grosseperis2011}).

\vskip 5pt

\noindent 2. It is known that the d-mixing property of $T_1,\dots,T_p$ implies that $c_1T_1,\dots,c_pT_p$ are also d-mixing for all unimodular scalars $c_1,\dots,c_p$ (cf.\,\cite[Remark 24(i)]{besmartinperis2011}). However, the corresponding result in case of s-mixing operators does not hold in general. Indeed, for a mixing operator $T$, the pair $T,T$ is clearly s-mixing, but $T,-T$ are not s-mixing any more. To see this, assume that $T,-T$ are s-mixing. Then Proposition \ref{Prop:s-transitive-s-mixing}(ii) would imply that $(T^{n_k}),((-T)^{n_k})$ are densely s-universal for each strictly increasing sequence $(n_k)$ in $\mathbb{N}$ -- but s-universality of the sequences $(T^{2k+1})$ and $((-T)^{2k+1})=(-T^{2k+1})$ is clearly not possible. In connection with this, it is stated in \cite[Remark 24(ii)]{besmartinperis2011} and actually proved in \cite[Proposition 4.9]{shkarin2008} that in case of unimodular scalars $c_1,\dots,c_p$ every d-hypercyclic vector $x_0$ for $T_1, \dots , T_p$ is also d-hypercyclic for $c_1T_1, \dots , c_pT_p$.
The proof uses crucially the fact that such a vector $x_0$ satisfies $(x_0, \dots ,x_0) \in HC(T_1 \oplus \cdots \oplus T_p)$. Thus, it cannot be adapted for s-hypercyclicity. Hence, we pose the question: Does the equality \,$s$-$HC(T_1, \dots ,T_p) =$ $s$-$HC(c_1T_1, \dots ,c_pT_p)$ \,hold?

\vskip 5pt

\noindent 3. Concerning again part (b) and regarding its proof, we may obtain a much stronger result in the case \,$\K = \C$ \,and \,$X$ \,a Banach space.
Recall that a nonempty subset \,$E \subset \C$ \,is said to be \textit{perfect} if it is closed and each point of \,$E$ \,is an accumulation point of \,$E$. In particular, every perfect set is uncountable. It is well known (see \cite[Theorem 8.138(b)]{bukovsky2011}) that there are perfect Dirichlet subsets of \,$\T$. We have that if \,$E \subset \T$
\,is a perfect Dirichlet set and \,$T \in L(X)$ \,is mixing, then the uncountable family of rotations \,$\{cT: \, c \in E \cup \{1\}\}$
is {\it densely uniformly s-hypercyclic}, in the sense that there is a dense set of vectors \,$x_0 \in X$ \,satisfying the following:
for every \,$y \in X$ \,there is \,$(n_k) \subset \N$ \,such that \,$\lim_{k \to \infty} \sup_{c \in E \cup \{1\}} \|(cT)^{n_k}x_0 - y\| = 0$.
Indeed, we can take a sequence \,$(m_k) \subset \N$ \,such that \,$\sup_{c \in E \cup \{1\}} |c^{m_k} - 1| = \sup_{c \in E} |c^{m_k} - 1| \to 0$
as \,$k \to \infty$. As \,$T$ \,is mixing, the set \,$HC((T^{m_k}))$ \,is dense. If \,$x_0 \in HC((T^{m_k}))$, then there is a subsequence \,$(n_k) \subset (m_k)$
\,with \,$T^{m_k}x_0 \to y$. The conclusion follows from the inequality \,$\|(cT)^{n_k}x_0 - y\| \le \|c^{n_k}(T^{n_k}x_0- y)\| + \|(c^{n_k} - 1)y\|$.

\vskip 5pt

\noindent 4. Proposition \ref{Prop: multiples} furnishes examples of pairs of operators --on spaces of sequences or of holomorphic functions (see sections 5--6)-- that are s-hypercyclic but not d-hypercyclic: the multiples $2B,-2B$ of the backward shift $B$ on $\ell_q$ $(1 \le q < \infty )$ or $c_0$; $D,-D$ on $H(\C )$
($Df := f'$); $C_\varphi, -C_\varphi$ on $H(G)$, where $C_\varphi f :=f\circ\varphi$, $G \subset \C$ is a simply connected domain and $\varphi$ is a run-away automorphism of \,$G$.
\end{remarks}

\section{Backward shifts and s-hypercyclicity}

\quad In this section, we consider the sequence spaces $c_0$ and $\ell_q$ $(1 \le q < \infty )$
over $\K = \R$ or $\C$.
If $a = (a_n)_{n\in\mathbb{N}}$ is a bounded sequence in $\K \setminus \{0\}$, then $B_a$ will denote the weighted backward shift
$$B_a : (x_0,x_1,x_2, \dots ) \in X \mapsto (a_1x_1,a_2x_2, \dots ) \in X$$
on $X = c_0$ or $\ell_q$. The unweighted backward shift $B$ is $B = B_a$, where $a = (1,1,1, \dots )$.
Salas characterized the hypercyclicity of $B_a$ in terms of the weight sequence $a$.
B\`es and Peris \cite[Theorem 4.1]{besperis2007} did the same for the d-hypercyclicity of different powers of $B_a$.
This characterization happens to hold also for s-hypercyclicity.

\begin{proposition} \label{Prop: s-h powers of wbs}
Let \,$X = c_0$ \,or \,$\ell_q\,\,(1 \le q < \infty)$, $p\ge2$ and let \,$r_1, \dots , r_p \in \N$ \,with \,$r_1 < r_2 < \cdots < r_p$ \,be given.
For each \,$l \in \{1, \dots ,p\}$, let \,$a_l = (a_{l,n})_{n\in\mathbb{N}}$ be a weight sequence. Then the following are equivalent:
\begin{enumerate}
\item[\rm (i)] $B_{a_1}^{r_1}, \dots , B_{a_p}^{r_p}$ are d-hypercyclic.
\item[\rm (ii)] $B_{a_1}^{r_1}, \dots , B_{a_p}^{r_p}$ are s-hypercyclic.
\item[\rm (iii)] For every \,$M > 0$ \,and every \,$k \in \N$ \,there is \,$m \in \N$ \,satisfying, for each \,$j \in \{0,1, \dots ,k\}$, that
$|a_{l,j+1} \cdots a_{l,j+r_lm}| > M$ $(1 \le l \le p)$ and ${|a_{l,j+1} \,\cdots \,a_{l,j+r_lm}| \over |a_{s,j+(r_l-r_s)m+1} \,\cdots \,a_{s,j+r_lm}|} > M$
$(1 \le s < l \le p)$.
\item[\rm (iv)] $B_{a_1}^{r_1}, \dots , B_{a_p}^{r_p}$ satisfy the d-hypercyclicity criterion.
\item[\rm (v)] $B_{a_1}^{r_1}, \dots , B_{a_p}^{r_p}$ satisfy the s-hypercyclicity criterion.
\end{enumerate}
\end{proposition}

\begin{proof}
The equivalence of (i), (iii) and (iv) is proved in \cite[Theorem 4.1]{besperis2007}.
That (i) implies (ii) is trivial. Moreover, (ii) $\Rightarrow$ (iii) is proved in fact in the proof of ``(a) $\Rightarrow$ (b)''
of the same reference, since only the simultaneous approximation of one vector (namely $e_0 + \cdots + e_q$) is used. Finally, we clearly have (iv) $\Rightarrow$(v) $\Rightarrow$ (ii).
\end{proof}

\begin{remarks}
1. An analogous result about equivalence of d- and s-hypercyclicity also works for powers of weighted
bilateral shifts (see Theorem 4.7 of \cite{besperis2007} and its proof).

\vskip 5pt

\noindent 2. Corollary 4.4 in \cite{besperis2007} also works with just s-universality, as it is a consequence of Theorem 4.1 there.
In particular, we have that $B_a, B_a^2, \dots , B_a^p$ are s-hypercyclic on $X$ if and only if
$B_a \oplus B_a^2 \oplus \cdots \oplus B_a^p$ is hypercyclic on $X^p$. B\`es, Martin and Peris \cite[p.~855]{besmartinperis2011}
constructed an operator $T := B_a$ on $\ell_2$ such that $T$ is hypercyclic but $T \oplus T^2$ is not hypercyclic on
$\ell_2 \oplus \ell_2$, so that $T,T^2$ is  not d-hypercyclic on $\ell_2$. Then we obtain that $T,T^2$ are even {\it not s-hypercyclic.}
According to \cite[Theorem 4.8]{grosseperis2011}, the mentioned $T = B_a$ is not mixing. In \cite[Sect.~3]{besmartinperisshkarin2012},
a mixing operator $T \in L(\ell_2)$ for which $T,T^2$ are not d-mixing is exhibited. But the existence of a mixing $T$ on a separable Banach space such
that $T,T^2$ are not d-hypercyclic is unknown so far \cite[Question 3.7]{besmartinperisshkarin2012}.
\end{remarks}

A more delicate question arises when $r_1 \le r_2 \le \cdots \le r_p$. In \cite[Corollary 4.2]{besperis2007}, the following is proved
for weighted powers of the unweighted backward shift: if \mbox{$p \ge 2$} and $r_l \in \N$, $\la_l \in \K$ $(1 \le l \le p)$ with $r_1 \le r_2 \le \cdots \le r_p$, then $\la_1 B^{r_1}, \dots ,\la_p B^{r_p}$ are d-hypercyclic if and only if \,$r_1 < r_2 < \cdots < r_p$ \,and \,$1 < |\la_1| < |\la_2| < \cdots  < |\la_p|$. The following result shows that s-hypercyclicity is possible under slightly weaker assumptions.

\begin{proposition} \label{Prop: s-hc unweighted b.s.}
Let $p \ge 2$, and let $r_l \in \N$, $\la_l \in \K$ $(1 \le l \le p)$ with $r_1 \le r_2 \le \cdots \le r_p$.
Let $A$ denote the set $A := \{j \in \{1,\dots ,p-1\}: \, r_j = r_{j+1}\}$ and consider the conditions
\begin{enumerate}
\item[\rm (i)] $1 < |\la_j|$ for all $j \in \{1, \dots , p\}$,
\item[\rm (ii)] $|\la_j| < |\la_{j+1}|$ for all $j \in \{1, \dots , p-1\} \setminus A$,
\item[\rm (iii)] $|\la_j| = |\la_{j+1}|$ for all $j \in A$.
\end{enumerate}
Then \,$\la_1 B^{r_1}, \dots ,\la_p B^{r_p}$ are s-hypercyclic on $X = c_0$ or $\ell_q\,\,(1 \le q < \infty)$ if and only if \textnormal{(i)},\textnormal{(ii)} and \textnormal{(iii)} hold.
\end{proposition}

\begin{proof}
First, suppose that conditions (i),(ii) and (iii) hold. We write $\{1,\dots,p\}\backslash A = \{t_1, \dots , t_d\}$, with $d \in \N$ and $t_1 < \cdots < t_d$.
As the set $\{\la_i/\la_j: \, i,j \in \{1, \dots ,p\}$ with $|\la_i| = |\la_j|\} \subset \T$ is finite, it is a Dirichlet set.
Hence there exists a strictly increasing sequence $(n_k) \subset \N$ such that
$$
\left( {\la_i \over \la_j} \right)^{n_k}  \to 1 \,\, (k \to \infty ) \hbox{ \ for all } i,j \in \{1, \dots ,p\} \hbox{\, with \,} |\la_i| = |\la_j|. \eqno (1)
$$
Consider the set $X_0$ of finite sequences, that is, $X_0 := c_{00} = \{x = (x_n) \in X:$ exists $n_0 = n_0(x) \in \N$ such that $x_n = 0$ for all $n\ge n_0\}$.
Then $X_0$ is dense in $X$. If we set $W_0 := \Delta (X_0^p) \subset X^p$, then $\ovl{W_0} = \ovl{\Delta (X_0^p)} \supset \Delta (X^p)$ because $X_0$ is dense in $X$. Now, we set $T_j := \la_j B^{r_j}$ $(j=1, \dots ,p)$. Define, for each $k \in \N$, the mapping $R_k:W_0 \to X$ as follows. If $x = (x_1, x_2, \dots ,x_N,0,0,0, \dots ) \in X_0$ and $w = (x,x, \dots ,x)$, then
\begin{equation*}
	R_k w =
	\left\{
	\begin{array}{lll}
		(0_1,u_1,0_2,u_2, \dots ,0_N,u_N,0,0,0, \dots ) & \text{if} & n_k \ge N \\
		(0,0,0, \dots ) & \text{if} & n_k < N,
	\end{array}
	\right.
\end{equation*}
where $0_1 :=(0,0, \dots ,0)$ [$r_{t_1} n_k$ times], $0_l :=(0,0, \dots ,0)$ [$(r_{t_l} - r_{t_{l-1}}) n_k - N$ times] if $l \ge 2$ and $u_l := \big( {1 \over \la_{t_l}^{n_k}} x_1, \dots , {1 \over \la_{t_l}^{n_k}} x_N \big)$ $(l \ge 1)$. We have:
\begin{enumerate}
\item[\rm (a)] For each $j \in \{1, \dots ,p\}$ and each $x = (x_1, x_2, \dots ,x_N,0,0,0, \dots ) \in X_0$, $T_j^{n_k}x = 0$
as soon as $r_j n_k > N$, so $T_j^{n_k} \to 0$ $(k \to \infty )$ pointwise on $X_0$.
\item[\rm (b)] For every $w = (x, \dots ,x) \in W_0$ as before, the definition of \,$R_k$ \,together with (i) yields $R_k w \to 0$ as $k \to \infty$.
\item[\rm (c)] Fix $w = (x, \dots ,x) \in W_0$, where $x = (x_1, x_2, \dots ,x_N,0,0,0, \dots )$. For every $j \in \{1, \dots ,p\}$ there is exactly
one $l \in \{1, \dots ,d\}$ such that $|\la_j| = |\la_{t_l}|$, due to (ii) and (iii). Finally, if $n_k \ge N$, we have
\begin{equation*}
\begin{split}
T_j^{n_k} R_k w = \bigg(&\big({\la_j \over \la_{t_l}} \big)^{n_k} x_1, \big({\la_j \over \la_{t_l}} \big)^{n_k} x_2, \dots , \big({\la_j \over \la_{t_l}} \big)^{n_k} x_N, 0,0, \dots ,0, \\
&\big({\la_j \over \la_{t_{l+1}}} \big)^{n_k} x_1, \big({\la_j \over \la_{t_{l+1}}} \big)^{n_k} x_2, \dots ,\big({\la_j \over \la_{t_{l+1}}} \big)^{n_k} x_N, 0, 0, \dots ,0, \dots , \\
&\big({\la_j \over \la_{t_d}} \big)^{n_k} x_1, \big({\la_j \over \la_{t_d}} \big)^{n_k} x_2, \dots ,
\big({\la_j \over \la_{t_d}} \big)^{n_k} x_N, 0,0,0,0, \dots \bigg) .
\end{split}
\end{equation*}
%Note that $T_j^{n_k} R_k w$ is a finite sequence whose length does not depend on $k$.
It follows from (ii) that $({\la_j \over \la_{t_s}})^{n_k} x_\nu \to 0$ as $k \to \infty$ for all $s \in \{l+1, \dots ,d\}$ and all $\nu \in \{1, \dots ,N\}$,
while (1) entails that $({\la_j \over \la_{t_l}})^{n_k} x_\nu \to x_\nu$ as $k \to \infty$ for all $\nu \in \{1, \dots ,N\}$. Consequently,
$T_j^{n_k} R_k w \to (x_1, x_2, \dots ,x_N,0,0,0, \dots ) = x$.
\end{enumerate}
An application of the s-hypercyclicity criterion (see also Remark \ref{Remarks Delta + disjoint}.1) concludes the first part of the proof.

Now, suppose that $\la_1 B^{r_1}, \dots ,\la_p B^{r_p}$ are s-hypercyclic. Since hypercyclic operators on normed spaces have norm larger than 1, we obtain
$$1<\|\la_jB^{r_j}\|=|\la_j|\|B^{r_j}\|=|\la_j|$$
for all $j=1,\dots,p$ (cf.\,\,the proof of Corollary 4.2 in \cite{besperis2007}), i.e.\,\,condition (i) holds. For each $j\in\{1,\dots,p-1\}\backslash A$, we have $r_j<r_{j+1}$. Hence, as $\la_jB^{r_j},\la_{j+1}B^{r_{j+1}}$ are s-hypercyclic, Proposition \ref{Prop: s-h powers of wbs}, (ii) $\Rightarrow$ (iii), and the same approach as in the proof of Corollary 4.2 in \cite{besperis2007} yield $|\la_j|<|\la_{j+1}|$, i.e.\,\,condition (ii) holds. Finally, for each $j\in A$, we have $r_j=r_{j+1}$. Hence, the s-hypercyclicity of $$\la_jB^{r_j},\,\,\la_{j+1}B^{r_{j+1}}=\frac{\la_{j+1}}{\la_j}\cdot\la_jB^{r_j}$$ implies $|\la_{j+1}\slash\la_j|=1$ (see Proposition \ref{Prop: multiples}(a)) and thus $|\la_j|=|\la_{j+1}|$, i.e.\,\,condition (iii) holds.
\end{proof}

For instance, the operators \,$2B, 3B^2, -3B^2$, being not d-hypercyclic, are s-hypercyclic. Further study of d-hypercyclicity of weighted unilateral and bilateral backward shifts can be found in \cite{besmartinsanders2014}.

\section{s-hypercyclicity in spaces of holomorphic functions}

\quad Let $G \subset \C$ be a domain, that is, a nonempty connected open subset of $\C$. We endow the space $H(G)$ of all holomorphic
(or analytic) functions $G \to \C$ with the topology of uniform convergence on compacta, so that $H(G)$ becomes a separable Fr\'echet space.
In this section we are concerned with s-hypercyclicity of finite sets of operators on $H(G)$ (or on subspaces of it) for certain domains $G$.

\vskip .15cm

Recall that if $X$ is a topological vector space and $T \in L(X)$, then $T$ is said to be \textit{supercyclic} provided that there exists some $x_0  \in X$ whose
projective orbit $\{\la T^n x_0: \, n\in\N, \, \la \in \K\}$ is dense in $X$. If $T_1, \dots ,T_p \in L(X)$, they are called \textit{d-supercyclic} (see \cite{besmartinperis2011}) if there is $x_0 \in X$ such that
$\{\la [T_1^n, \dots ,T_p^n]x_0: \, n \in \N, \, \la \in \K\}$ is dense in $X^p$. Consistently, we say that
$T_1, \dots ,T_p$ are {\it s-supercyclic} whenever
$\ovl{\{\la [T_1^n, \dots ,T_p^n]x_0: \, n \in \N, \, \la \in \K\}}  \supset\Delta (X^p)$.

\vskip .15cm

Let $LFT(\D )$ denote the family of all linear fractional transformations $\varphi (z) = {az+b \over cz+d}$ of the
complex plane such that $\varphi (\D ) \subset \D$. The subfamily Aut$(\D)$ of automorphisms of $\D$ consists of all
onto members of $LFT(\D )$. See e.g.~\cite[Chapter 1]{shapiro1993} for terminology related to these families.
If $\nu \in \R$, then $S_\nu$ denotes the weighted Hardy space
$S_\nu = \{f(z) = \sum_{n \ge 0} a_nz^n \in H(\D ): \, \|f\| := (\sum_{n \ge 0} |a_n|^2(n+1)^{2\nu})^{1/2} < \infty \}$. Each $S_\nu$ is a Hilbert space,
and the choices $\nu = -1/2,0,1/2$ correspond, res\-pec\-ti\-ve\-ly, to the classical Bergman, Hardy and Dirichlet spaces.
Thanks to the results in \cite{besmartinperis2011}, we obtain without effort the next two assertions.

\begin{proposition} \label{Prop: composition on H(D)}
Let $\varphi_1, \dots , \varphi_p \in LFT(\D )$ pairwise distinct. Then the following are equivalent:
\begin{enumerate}
\item[\rm (a)] $C_{\varphi_1}, \dots , C_{\varphi_p}$ are s-supercyclic on $H(\D )$.
\item[\rm (b)] $\mu_1 C_{\varphi_1}, \dots , \mu_p C_{\varphi_p}$ are s-mixing on $H(\D )$ for all nonzero scalars $\mu_1, \dots ,\mu_p$.
\item[\rm (c)] $C_{\varphi_1}, \dots , C_{\varphi_p}$ are d-supercyclic on $H(\D )$.
\item[\rm (d)] $\mu_1 C_{\varphi_1}, \dots , \mu_p C_{\varphi_p}$ are d-mixing on $H(\D )$ for all nonzero scalars $\mu_1, \dots ,\mu_p$.
\item[\rm (e)] $\varphi_1 \dots , \varphi_p$ have no fixed point in $\D$, and satisfy that if any two $\varphi_l, \varphi_j$ have the same attractive
fixed point $\alpha$, then \,$\varphi_l ' (\alpha ) = \varphi_j ' (\alpha ) < 1$ \,is not possible.
\end{enumerate}
\end{proposition}

\begin{proof}
The equivalence of (c), (d) and (e) is proved in \cite[Theorem 4]{besmartinperis2011}.
The implications (d)  $\Rightarrow$ (b)  $\Rightarrow$ (a) are trivial. Finally, (a) $\Rightarrow$ (e) is proved in fact in the proof of
Theorem 4 in \cite{besmartinperis2011}. Indeed, it is used there a result (Lemma 14 in \cite{besmartinperis2011}) asserting that
if $\varphi_1,\varphi_2 \in LFT(\D )$ are hyperbolic and share an attractive fixed point $\alpha$ with $\varphi_1 ' (\alpha ) = \varphi_2 ' (\alpha )$, then $C_{\varphi_1}, C_{\varphi_2}$
are not d-supercyclic on $H(\D )$. But a closer look at its proof shows that $C_{\varphi_1}, C_{\varphi_2}$ are in fact even not s-supercyclic;
indeed, via contradiction, only one function $g$ \,is assumed to be simultaneously approximated by projective orbits.
\end{proof}

\begin{proposition}
Let $\varphi_1, \dots , \varphi_p \in LFT(\D )$ pairwise distinct and let $\nu < 1/2$. Then the following are equivalent:
\begin{enumerate}
\item[\rm (a)] $C_{\varphi_1}, \dots , C_{\varphi_p}$ are s-supercyclic on $S_\nu$.
\item[\rm (b)] $C_{\varphi_1}, \dots , C_{\varphi_p}$ are s-mixing on $S_\nu$.
\item[\rm (c)] $C_{\varphi_1}, \dots , C_{\varphi_p}$ are d-supercyclic on $S_\nu$.
\item[\rm (d)] $C_{\varphi_1}, \dots , C_{\varphi_p}$ are d-mixing on $S_\nu$.
\item[\rm (e)] Each $\varphi_l$ is a parabolic automorphism or a hyperbolic map without
               fixed points in $\D$, and there are no two $\varphi_l, \varphi_j$ having a common fixed point $\alpha$
               such that \,$\varphi_l ' (\alpha ) = \varphi_j ' (\alpha ) < 1$.
\end{enumerate}
\end{proposition}

\begin{proof}
The equivalence of (c), (d) and (e) is proved in \cite[Theorem 3]{besmartinperis2011}.
The implications (d)  $\Rightarrow$ (b)  $\Rightarrow$ (a) are trivial.
As for (a) $\Rightarrow$ (e), observe that in the proof of Theorem 3 in \cite{besmartinperis2011}, only the supercyclicity of
each $C_{\varphi_l}$ is necessary for the first assertion in (e) and that the Comparison Principle \cite[Proposition 8]{besmartinperis2011}
--that also works for s-supercyclicity-- implies that $C_{\varphi_1}, \dots , C_{\varphi_p}$ are s-supercyclic on $H(\D )$.
Now, the second assertion of (e) follows from Proposition \ref{Prop: composition on H(D)}.
\end{proof}

\begin{remarks}
1. Recall that if $X$ is an F-space and $T \in L(X)$ is invertible and hypercyclic, then $T^{-1}$ is also hypercyclic.
Analogously as in Example 22 in \cite{besmartinperis2011}, by combining the preceding two propositions, we obtain that there are
hyperbolic $\varphi_1, \varphi_2 \in {\rm Aut} (\D )$ such that $C_{\varphi_1}, C_{\varphi_2}$ are d-hypercyclic (so s-hypercyclic)
on $H^2(\D )$ (the Hardy space) and on $H(\D )$, and
$C_{\varphi_1^{-1}} = (C_{\varphi_1})^{-1}, C_{\varphi_2^{-1}} = (C_{\varphi_2})^{-1}$ are even not s-supercyclic on $H^2(\D )$ or $H(\D )$
(note that $\varphi_1^{-1}$ and $\varphi_2^{-1}$ are also hyperbolic).
Hence, in general, the d-hypercyclicity of $T_1, \dots ,T_p$ does {\it not} imply the s-hypercyclicity of $T_1^{-1}, \dots , T_p^{-1}$ if
\,$T_1, \dots ,T_p$ are invertible. Moreover, finitely many composition operators generated by non-elliptic automorphisms of $\D$ may be {\it not}
s-hypercyclic on $H(\D )$ or on $H^2(\D )$.

\vskip 5pt

\noindent 2. Further study of d-hypercyclicity of composition operators, this time on weighted Bergman spaces on $\D$, is performed in \cite{zhangzhou2016}.
\end{remarks}

In 1929 Birkhoff \cite{birkhoff1929} proved that the translation operator $\tau_a$ $(a \in \C \setminus \{0\})$
given by $(\tau_a f)(z) = f(z+a)$ is hypercyclic on the space $H(\C)$ of entire functions. It is proved in \cite[Prop.~5.5]{bernal2007}
and \cite[Theorem 3.1]{besperis2007} that if $a_1, \dots , a_p$ are pairwise distinct nonzero complex numbers, then $\tau_{a_1}, \dots ,\tau_{a_p}$ are d-hypercyclic. Trivially, we obtain: if $a_1, \dots ,a_p \in \C \setminus \{0\}$, then $\tau_{a_1}, \dots ,\tau_{a_p}$ are s-hypercyclic.
As the next proposition shows, we may obtain a slight extension to weighted translation operators. %Before this, we need the following lemma showing the existence of special entire functions.
%
%\begin{lemma} \label{Lemma: entire fast decreasing}
%Let $a,b,c \in \R$ such that $0 < b - a < 2\pi$ and $c > 0$, and $S_{a,b,c} := \{re^{i \theta} \in \C : \,r \ge c$ and $a \le \theta \le b\}$.
%Then there exists an entire function \,$\varphi$ \,such that
%$$
%\varphi (0) = 1 \hbox{ \ and \ } \lim_{z \to \infty \atop z \in S_{a,b,c}} z^m \, \varphi (z) = 0 \hbox{ \ for all } \, m \in \N_0.
%$$
%\end{lemma}
%
%\begin{proof}
%We are going to use an important tangential approximation theorem due to Arakelian which can be found, for instance, in \cite[pp.~161--162]{gaier1987}.
%Consider the set \,$F := S_{a,b,c} \cup \{0\}$. Then \,$F$ \,is a closed subset of \,$\C$ \,such that \,$\C_\infty \setminus F$ \,connected and locally connected in \,$\C_\infty$. Then, for every continuous function \,$f : F \to \C$ \,that is holomorphic in \,$F^0$ (= the interior of $F$) and every continuous function \,$\epsilon : [0,+\infty ) \to (0,+\infty )$, there is \,$g \in H(\C )$ \,such that \,$...$ ***************
%... ... APLICAR ARAKELIAN-APROX.TANGENCIAL A LA UNION DE S CON EL 0, Y DIVIDIR POR PHI(0). \cite[pp.~161--162]{gaier1987} *********** ... ...
%\end{proof}

\begin{proposition}
Let \,$p \ge 2$, and let \,$a_1, \dots, a_p,\la_1, \dots , \la_p \in \C \setminus \{0\}$
\,such that \,$|\la_j| = |\la_l|$ \,for all \,$j,l \in \{1, \dots ,p\}$ \,with \,$a_j = a_l$. Then
there is a sequence \,$(n_k) \subset \N$ \,such that
the sequences \,$(\la_1 \tau_{a_1})^{n_k}, \dots , (\la_p \tau_{a_p})^{n_k}$ \,are s-mixing. In particular,
the operators \,$\la_1 \tau_{a_1}, \dots , \la_p \tau_{a_p}$ \,are densely s-hypercyclic on \,$H(\C )$.
\end{proposition}

\begin{proof}
Select a finite sequence $\{j(1) < j(2) \cdots < j(q)\} \subset \{1, \dots ,p\}$ satisfying that, if \,$b_l := a_{j(l)}$
$(l=1, \dots ,q)$, then the \,$b_l$'s \,are pairwise distinct and \,$\{a_1, \dots ,a_p\} = \{b_1, \dots ,b_q\}$. Let \,$\mu_l := \la_{j(l)}$.
Consider the operators \,$T_j := \la_j \tau_{a_j}$ $(j=1, \dots ,p)$
\,and \,$S_l := T_{j(l)} = \mu_l \tau_{b_l}$ $(l=1, \dots ,q)$.

\vskip 6pt

Let us prove that $S_1, \dots ,S_q$ are s-mixing. In fact, by following the approach of the proof of \cite[Theorem 3.1]{besperis2007},
we can prove that they are even d-mixing. To this end, and taking into account that the sets \,$V(h,r,\ve ) := \{f \in H(\C ) : \,|f(z)-h(z)| < \ve$ for all
$z \in \ovl{B}(0,r)\}$ $(h \in H(\C), \, \ve > 0, \, r > 0)$, form a basis for the topology of $H(\C )$, it is enough to prove that, for given
\,$h,g_1, \dots ,g_q \in H(\C )$ \,and \,$\ve , r > 0$, there is \,$n_0 \in \N$ \,such that, for every \,$n \ge n_0$, there
exists an entire function \,$f$ \,with
$$
|f(z) - h(z)| < \ve \hbox{ and } |(S_l^n f)(z) - g_l(z)| < \ve \quad (z \in \ovl{B}(0,r), \,l=1, \dots ,q). \eqno (1)
$$
Select \,$n_0 \in \N$ \,with \,$n_0 > \max_{i \ne l} {2r \over |b_i - b_l|} + \max_{1 \le l \le q} {2r \over |b_l|}$.
Then, for each $n \ge n_0$, the disks \,$\ovl{B}(0,r), \, \ovl{B}(nb_1,r), \dots , \, \ovl{B}(nb_q,r)$ \,are pairwise disjoint.
Pick $s > r$ such that the disks \,$\ovl{B}(0,s), \, \ovl{B}(nb_1,s), \dots , \, \ovl{B}(nb_q,s)$ \,are still pairwise disjoint.
Let \,$K := \ovl{B}(0,r) \cup \ovl{B}(nb_1,r) \cup \cdots \cup \ovl{B}(nb_q,r)$
\,and \,$\Omega := {B}(0,s) \cup B(nb_1,s) \cup \cdots \cup B(nb_q,s)$.
Note that $\Omega$ is an open set, $\Omega \supset K$ \,and \,$K$ \,is a compact subset having connected complement.
Consider the function \,$F: \Omega \to \C$ \,defined by
$$
F(z) = h(z) \hbox{ if } z \in B(0,s) \hbox{ \ and \ } F (z) := \mu_l^{-n} g_l(z - nb_l) \hbox{ if } z \in B(nb_l,s) \,\, (1 \le l \le q).
$$
Then \,$F \in H(\Omega)$. From Runge's approximation theorem (see e.g.~\cite{gaier1987}), it follows that there exists a
polynomial \,$f$ (so $f \in H(\C )$) such that \,$|f(z) - F(z)| < \ve /(1 + |\mu^n_l|)$ for all $z \in K$. But this implies that
\,$|f(z) - h(z)| < \ve$ \,on \,$\ovl{B}(0,r)$ \,and \,$|\mu_l^n f(z) - g_l(z - nb_l)| < \ve$ \,on \,$\ovl{B}(nb_l,r)$. Since the last inequality
is equivalent to \,$|\mu_l^n f(z + nb_l) - g_l(z)| < \ve$ \,on \,$\ovl{B}(0,r)$, (1) is obtained.

\vskip 6pt

As the set \,$D := \{\la_j/\la_l: \, j,l \in \{1, \dots ,p\}$ with $a_j = a_l \} \subset \T$ \,is finite, it is a Dirichlet set.
Then there is a strictly increasing sequence \,$(n_k) \subset \N$ \,such that \,$\xi^{n_k} \to 1$ as $k \to \infty$,
for all $\xi \in D$.

\vskip 6pt

Fix a subsequence \,$(m_k)$ \,of \,$(n_k)$. Since \,$S_1, \dots ,S_q$ \,are s-mixing, the set \break
s-$HC((S_1^{m_k}), \dots ,(S_q^{m_k}))$ \,is dense (see Proposition \ref{Prop:s-transitive-s-mixing}).
Fix \,$f$ \,in \break
s-$HC((S_1^{m_k}), \dots ,(S_q^{m_k}))$. For each \,$\nu \in \{1, \dots ,p\}$ \,there is a unique \,$l = l(\nu ) \in \{1, \dots ,q\}$
\,such that \,$a_\nu = b_l$, so that \,$|\la_\nu | = |\mu_l|$. Observe that \,$\xi_\nu := \la_\nu /\mu_l \in D$. Then \,$\xi_\nu^{n_k} \to 1$, hence \,$\xi_\nu^{m_k} \to 1$ $(k \to \infty )$ \,for all \,$\nu \in \{1, \dots ,p\}$. Given \,$g \in H(\C )$, we can find a subsequence \,$(p_k)$ \,of \,$(m_k)$ with \,$S_{l(\nu )}^{p_k} f \to g$ $(k \to \infty )$ \,uniformly on compacta for every \,$\nu \in \{1, \dots , p\}$. Since  \,$\xi_\nu^{p_k} \to 1$ for all $\nu$, we obtain that
\,$T_\nu^{p_k} f = \xi_\nu^{p_k} S_{l(\nu )}^{p_k} f \longrightarrow 1 \cdot g = g$ \,$(k \to \infty)$ \,uniformly on compacta for every \,$\nu = 1, \dots ,p$.
Therefore \,$f \in$ s-$HC((T_1^{m_k}), \dots ,(T_p^{m_k}))$, which shows that this set is dense. By Proposition \ref{Prop:s-transitive-s-mixing}, the sequences
$(T_1^{n_k}), \dots ,(T_p^{n_k})$ are s-mixing, as required.
\end{proof}

Another important collection of operators on $H(\C )$ is that of differentiation operators.
Consider the derivative operator $D : f \in H(\C ) \mapsto f' \in H(\C )$. Its hypercyclicity on $H(\C )$ was proved by
MacLane in 1952 \cite{maclane1952}. It is shown in \cite[Prop.~3.3]{besperis2007} that if $p \ge 2$, $r_1,\dots,r_p\in\mathbb{N}$ with $r_1 < \cdots < r_p$ and
$\la_1, \dots , \la_p \in \C \setminus \{0\}$, then $\la_1 D^{r_1}, \dots , \la_p D^{r_p}$ are d-mixing, so densely d-hypercyclic.
Concerning s-hypercyclicity, the following proposition shows that somewhat softer assumptions are allowed, although, similarly to the last proposition,
we have not been able to obtain the s-mixing property for the whole sequences.

\vskip .15cm

\begin{proposition}
Let \,$r_1 \le \cdots \le r_p$ \,be positive integers and \,$\la_1,  \dots , \la_p \in \C \setminus \{0\}$, where $p \ge 2$.
Suppose that \,$|\la_j| = |\la_l|$ \,for all $j,l \in \{1, \dots ,p\}$ \,with \,$r_j = r_l$.

\vskip 3pt

Then there is a sequence $(n_k) \subset \N$ such that the sequences $(\la_1 D^{r_1})^{n_k}, \dots , (\la_p D^{r_p})^{n_k}$ are s-mixing. In particular, the operators
$\la_1 D^{r_1}, \dots , \la_p D^{r_p}$ are densely s-hypercyclic on $H(\C )$.
\end{proposition}

\begin{proof}
As the set $\{\la_j/\la_l: \, j,l \in \{1, \dots ,p\}$ with $r_j = r_l \} \subset \T$ is finite, it is a Dirichlet set.
Then there is a strictly increasing sequence $(n_k) \subset \N$ such that $(\la_j/\la_l)^{n_k} \to 1$ as $k \to \infty$,
for all $j,l \in \{1, \dots ,p\}$ with $r_j = r_l$. Put $X_0:= \{$polynomials$\} = {\rm span} \{z^m: \, m\in\N_0\}$ and
$W_0 := \Delta (X_0^p)$. Then $X_0$ is dense in $X := H(\C )$ and $\ovl{W_0} = \ovl{\Delta (X_0^p)} \supset \Delta (X^p)$.
Let $T_j := \la_j D^{r_j}$ $(1 \le j \le p)$. For each $k \in \N$, define the map $R_k : W_0 \to X$ via %\hfil\break \phantom{aaaaaaaaaaa}
$$
R_k (z^m, \dots ,z^m) := \sum_{l=1}^p {1 \over \tau(l)} \cdot {1 \over \la_l^{n_k}}
\cdot {z^{m + r_l n_k} \over (m+1)(m+2) \cdots (m + r_l n_k)},
$$ %\hfil\break
where $\tau (l) := {\rm card} \, \{i \in \{1, \dots, p\}: \, r_i = r_l\}$ $(1 \le l \le p)$. Then $R_k$ is extended to the whole $W_0$ by linearity.
We have:
\begin{enumerate}
\item[(i)] $T_j^{n_k} z^m = 0$ as soon as $n_k r_j > m$, so $T_j^{n_k} z^m \to 0$ as $k \to \infty$ for all $j \in \{1, \dots ,p\}$ and all $m \ge 0$.
Therefore, by linearity, $T_j^{n_k} \to 0$ $(k \to \infty )$ on $X_0$ for all  $j \in \{1, \dots ,p\}$.
\item[(ii)] Fix $m \in \N_0$ and a compact set $K \subset \C$. There is $M \in (0,+\infty )$ with $K \subset \ovl{B}(0,M)$. Given $k \in \N$, we obtain
\begin{equation*}
\begin{split}
\sup_{z \in K} |R_k (z^m, \dots ,z^m)| &\le \sum_{l=1}^p {1 \over \tau (l)} \cdot {1 \over \la_l^{n_k}} \cdot
{M^{m + r_l n_k} \over (m+1)(m+2) \cdots (m + r_l n_k)} \\
&\le \sum_{l=1}^p {1 \over \tau (l)} {M^{m + r_l n_k} / \la_l^{n_k} \over (m+1)(m+2) \cdots (m + n_k)} \\
&\le \sum_{l=1}^p {M^m \over \tau (l)} {(M^{r_l} / \la_l)^{n_k} \over n_k!} \to 0 \ (k \to \infty )
\end{split}
\end{equation*}
Hence, by linearity, $R_k \to 0$ $(k \to \infty )$ pointwise on $W_0$.
\item[(iii)] Fix $m \in \N_0$, $j \in \{1, \dots ,p\}$ and $k \in \N$ with $n_k > m$. Let us compute the
action of $T_j^{n_k}R_k$ on each $(z^m, \dots ,z^m)$. This yields three sums, the first of them corresponding to those $l \in \{1, \dots ,p\}$ with $r_l < r_j$, that equals $0$. Therefore
\begin{equation*}
\begin{split}
T_j^{n_k}&R_k (z^m, \dots ,z^m) = 0 + \sum_{l=1 \atop r_l = r_j}^p {1 \over \tau (l)} \cdot \big({\la_j \over \la_l} \big)^{n_k} \cdot {z^m} \\
                               &+ \sum_{l=1 \atop r_l > r_j}^p {1 \over \tau (l)} \cdot \big( {\la_j \over \la_l} \big)^{n_k} \cdot {z^{m + (r_l - r_j) n_k} \over (m + 1) (m + 2) \cdots (m + (r_l  - r_j)n_k)} \\
                               &\longrightarrow {1 \over \tau (j)} \cdot z^m \cdot \sum_{l=1 \atop r_l = r_j}^p 1 + 0 = z^m \ \ (k \to \infty )
\end{split}
\end{equation*}
uniformly on compacta in $\C$, because $\tau(j)=\tau(l)$ and $({\la_j \over \la_l})^{n_k} \to 1$ for all $(j,l)$ with $r_j = r_l$. By linearity again, we get $T_j^{n_k} R_k (w, \dots , w) \to w$ for all $j=1, \dots ,p$ and all $(w, \dots ,w) \in W_0$.
\end{enumerate}
The conclusion now follows from Theorem \ref{Prop: s-hc Delta-criterion} (or from Remark \ref{Remarks Delta + disjoint}.1).
\end{proof}

For instance, the operators $5D, D^2, -D^2, e^i D^2, {1 \over 10} D^3, -3D^4$ are s-hypercyclic, but clearly not d-hypercyclic.

\vskip .10cm

An extension unifying both Birkhoff's and MacLane's theorems takes place
by considering convolution operators on $H(\C )$, that is, operators commuting with all translations $\tau_a$.
Let $\Phi (z) = \sum_{n=0}^\infty a_n z^n \in H(\C )$. Then
$\Phi$ is said to be of \textit{exponential type} provided that there are positive
constants $A, \, B$ such that $|\Phi (z)| \leq A \exp (B|z|)$ for
all $z \in \C$. Then its associated differential operator $\Phi (D) =
\sum_{n=0}^\infty a_n D^n$ given by $\Phi (D) f =
\sum_{n=0}^\infty a_n f^{(n)}$ $(f \in H(\C ))$ defines an ope\-ra\-tor on $H(\C )$. Moreover, an operator
$T \in L(H(\C ))$ is of convolution if and only if $T = \Phi (D)$ for some entire function $\Phi$ of exponential type.
Note that $D$ and $\tau_a$ are special cases (take $\Phi (z) \equiv z$ and $\Phi (z) \equiv e^{az}$, resp.).
Godefroy and Shapiro \cite{godefroyshapiro1991} proved in 1991 that any nonscalar convolution operator is hypercyclic.
If $G$ is any domain in $\C$, then $\Phi (D)$ is also an operator on $H(G)$ whenever $\Phi$ is of \textit{subexponential type}, that is,
for given $\ve > 0$ there is a constant $A > 0$ such that
$|\Phi (z) | \leq A \exp (\ve |z|)$ for all $z \in \C$. We have that also $\Phi (D)$ is hypercyclic on $H(G)$ provided that
$G$ is simply connected (i.e.~its complement with respect to the one-point compactification $\C_\infty$ of $\C$ is connected) and $\Phi$ is not constant.
For s-hypercyclicity, we present the following assertion, with which we put an end to this introductory paper on s-universality.

\begin{proposition} \label{Prop: s-h Phi(D)}
Assume that $G \subset \C$ is a simply connected domain and that $\Phi_1, \dots ,\Phi_p$ are entire functions of subexponential type
{\rm (}or just of exponential type if $G = \C${\rm )}. Assume also that the set %\hfil\break \phantom{aaaaaaaaaaaaaaaaaa}
$$U_0 := \big\{ \la \in \C : \, \max_{1 \le j \le p} |\Phi_j(\la )| < 1 \big\}$$  %\hfil\break
is nonempty and that each set %\hfil\break \phantom{aaa}
$$U_i := \big\{\la \in \C : \,  |\Phi_i(\la )| > 1\, \textnormal{ and } \max_{1 \le j \le p} |\Phi_j(\la )| \le |\Phi_i(\la )|\big\} \ \ \ (1 \le i \le p)$$ %\hfil\break
has nonempty interior $U_i^0$. Suppose, in addition, that whenever $i,j \in \{1, \dots ,p\}$ satisfy \,$|\Phi_i (\la )| = |\Phi_j (\la )|$
\,for some $\la \in U_i^0$, there exists $\zeta \in \T$ with $\Phi_j = \zeta \cdot \Phi_i$.

\vskip 3pt

Then there is a sequence \,$(n_k) \subset \N$ \,such that
the sequences \,$(\Phi_1 (D))^{n_k}, \dots$ $\dots,(\Phi_p (D))^{n_k}$ are s-mixing. In particular, the operators
\,$\Phi_1 (D), \dots , \Phi_p (D)$ \,are densely s-hy\-per\-cy\-clic on $H(\C )$.
\end{proposition}

\begin{proof}
We write $e_\la := \exp ( \cdot \la )|_G$ for $\la \in \C$. It is easy to see that the functions $e_\la$ are linearly independent. Denote $V_i := U_i^0$ $(1 \le i \le p)$. As $U_0, V_1, \dots , V_p$ are open and nonempty, we obtain that \,$X_0 := {\rm span} \{e_\la : \, \la \in U_0\}$ \,is dense in \,$X := H(G)$ (because $G$ is simply connected: use Runge's approximation theorem together with the fact that ${\rm span}\{\exp ( \cdot \la ) : \, \la \in U_0\}$ is dense in $H(\C)$; see e.g.~\cite[Sect.~5]{godefroyshapiro1991}).
Hence \,$W_0 := \prod_{i=1}^p {\rm span} \{e_\la:\la \in V_i\}$ \,is dense in \,$X^p$.

\vskip 6pt

As $A := \{\zeta \in \T :$
exist $l,j \in \{1, \dots ,p\}$ with $\Phi_j = \zeta \Phi_l \} \subset \T$ is finite, it is a Dirichlet set; hence
there is a strictly increasing sequence $(n_k) \subset \N$ such that
$\zeta^{n_k} \to 1$ for all $\zeta \in A$.

\vskip 6pt

For each $i \in \{1, \dots ,p\}$, we put $T_i := \Phi_i(D)|_{H(G)}$, $E_i := \{j \in \{1, \dots ,p\}:$ exists $\zeta \in \T$ with $\Phi_j = \zeta \Phi_i\}$ and
$\tau (i) := {\rm card} (E_i)$. Notice that if $i \in E_j$, then $E_i = E_j$ (just use that $\T$ is a multiplicative group), hence $\tau (i) = \tau (j)$.
Given $i \in \{1, \dots ,p\}$ and $v_i \in{\rm span} \{e_\la:\la \in V_i\}$, there are uniquely determined scalars $c_{i,1}, \dots ,c_{i,J(i)} \in \C$ and
pairwise distinct $\la_{i,1}, \dots , \la_{i,J(i)} \in V_i$ such that $v_i = \sum_{l=1}^{J(i)} c_{i,l} e_{\la_{i,l}}$. For $k \in \N$ we define \,$R_k:W_0 \to X$ \,as
$$
R_k w := \sum_{i=1}^p {1 \over \tau (i)} \cdot \sum_{l=1}^{J(i)} {c_{i,l} \over \Phi_i (\la_{i,l})^{n_k}} \cdot e_{\la_{i,l}}, \eqno (1)
$$
where \,$w = (v_1, \dots ,v_p) \in W_0$ \,and the \,$v_i$'s are as above. We have:
\begin{enumerate}
\item[(i)] If $\la \in U_0$ and $j \in \{1, \dots ,p\}$, then $T_j^{n_k} e_\la = \Phi_j (\la )^{n_k} e_\la \to 0$ as
$k \to \infty$, because $|\Phi_j(\la )| < 1$. By linearity, we get $T_j^{n_k} \to 0$ on $X_0$.
\item[(ii)] Let \,$w = (v_1, \dots ,v_p) \in W_0$, so that \,$v_i = \sum_{l=1}^{J(i)} c_{i,l} e_{\la_{i,l}}$, as above.
Since \,$|\Phi_i (\la_{i,l})| > 1$, we get \,$|\Phi_i (\la_{i,l})^{n_k}| \to +\infty$ \,as \,$k \to \infty$, for each \,$i \in \{1, \dots ,p\}$ \,and each
$l=1, \dots ,J(i)$. From (1) one derives that \,$R_k w \to 0$.
\item[(iii)] Again, let \,$w = (v_1, \dots ,v_p) \in W_0$, with \,$v_i = \sum_{l=1}^{J(i)} c_{i,l} e_{\la_{i,l}}$.
Fix \,$j \in \{1, \dots ,p\}$ \,and \,$k \in \N$. We compute
$$
T_j^{n_k} R_k w = \sum_{i=1}^p {1 \over \tau (i)} \cdot \sum_{l=1}^{J(i)} {c_{i,l} \over \Phi_i (\la_{i,l})^{n_k}} \cdot T_j^{n_k} e_{\la_{i,l}}
$$
$$
\phantom{aaaaaaaaaaaaaaaaaaa} = \sum_{i=1}^p {1 \over \tau (i)} \cdot \sum_{l=1}^{J(i)} c_{i,l} \cdot  \left( {\Phi_j (\la_{i,l}) \over \Phi_i (\la_{i,l})} \right)^{n_k} e_{\la_{i,l}} = A_k + B_k,
$$
where $A_k$ ($B_k$, resp.) denotes the part of the preceding sum corresponding to those $i \in E_j$ ($i \not\in E_j$, resp.).
If $i \in E_j$, there is $\zeta = \zeta_{i,j} \in A$ such that $\Phi_j = \zeta \cdot \Phi_i$, so that
$\left( {\Phi_j (\la_{i,l}) \over \Phi_i (\la_{i,l})} \right)^{n_k} = \zeta^{n_k} \to 1$ as $k \to \infty$. Note that $\tau (i) = \tau (j)$ if $i \in E_j$.
Therefore, on the one hand, $$\qquad A_k \to \sum_{i=1 \atop i \in E_j}^p {1 \over \tau (i)} \cdot \sum_{l=1}^{J(i)} c_{i,l} \cdot e_{\la_{i,l}} =
{1 \over \tau (j)} \cdot \sum_{i=1 \atop i \in E_j}^p \sum_{l=1}^{J(i)} c_{i,l} \cdot e_{\la_{i,l}} = {1 \over \tau (j)} \cdot \sum_{i=1 \atop i \in E_j}^p v_i.$$
On the other hand, if $i \not\in E_j$, we have that $|\Phi_j(\la_{i,l})/\Phi_i(\la_{i,l})| < 1$ for all $l=1, \dots ,J(i)$
(indeed, as $\la_{i,l} \in V_i$, we have $|\Phi_j(\la_{i,l})| \le |\Phi_i(\la_{i,l})|$; if we assume $|\Phi_j(\la_{i,l})| = |\Phi_i(\la_{i,l})|$, then there would exist
$\zeta \in \T$ with $\Phi_j = \zeta \cdot \Phi_i$, which would yield $i \in E_j$, a contradiction). Hence $\big({\Phi_j(\la_{i,l}) \over \Phi_i(\la_{i,l})} \big)^{n_k} \to 0$, so
$B_k \to 0$. This entails $$T_j^{n_k} R_k w = A_k + B_k \to {1 \over \tau (j)} \cdot \sum_{i=1 \atop i \in E_j}^p v_i \,\,\, (k \to \infty ),$$ and the last vector belongs
to \,${\rm conv} (\{v_1, \dots ,v_p\})$ \,since in the last sum there are exactly \,$\tau (j)$ \,summands.
%To summarize, for every $w = (v_1, \dots ,v_p) \in W_0$ and every $j \in \{1, \dots ,p\}$, there is $y_j \in {\rm conv} (\{v_1, \dots ,v_p\})$
%such that $T_j^{n_k} R_k w \to y_j$.
\end{enumerate}
The conclusion follows, once again, from the s-hypercyclicity criterion (Theorem \ref{Prop: s-hc Delta-criterion}).
\end{proof}

\begin{remark}
Proposition 3.4 in \cite{besperis2007} (see also \cite[Theorem 5.3]{bernal2007}) asserts that if $U_0$ and $W_i := \{\la \in \C : \,  |\Phi_i(\la )| > 1$
and $\max_{j \ne i} |\Phi_j(\la )| < |\Phi_i(\la )|\}$ $(1 \le i \le p)$ are nonempty, then $\Phi_1(D), \dots , \Phi_p(D)$
are d-mixing. If these assumptions are satisfied, then the assumptions of Proposition \ref{Prop: s-h Phi(D)} are also
satisfied. Note that Proposition \ref{Prop: s-h Phi(D)} includes the case $\Phi_1 = \Phi$, $\Phi_j = c_j \Phi$ with $|c_j| = 1$ $(j=2, \dots ,p)$.
%In view of this, the assumptions of Proposition \ref{Prop: s-h Phi(D)} are ``optimal'', meaning that if there is
%$i \in \{1, \dots ,p\}$ such that the set $W_i$ is empty then, for every $\la \in \C$, either $|\Phi_i (\la )| \le 1$ or
%there exists $j \ne i$ with $|\Phi_j (\la )| \ge |\Phi_i(\la )|$.
%If the assumptions of Proposition \ref{Prop: s-h Phi(D)} are satisfied, there exists $z_i \in U_i^0$, so $|\Phi_i(z_i)| > 1$. Then there would exist
%$j \ne i$ such that $|\Phi_j (z_i )| \ge |\Phi_i(z_i )|$. But also $|\Phi_j (z_i )| \le |\Phi_i(z_i )|$. Hence $|\Phi_j (z_i )| = |\Phi_i(z_i )|$,
%so there is $\zeta \in \T$ with $\Phi_j = \zeta \Phi$.
\end{remark}

\vspace{6pt}

This paper does not intend to be exhaustive. Of course, many more sets of operators or of sequences of operators may be analyzed under the point of view of s-universality/s-hypercyclicity.
For instance, consider a compact set $K \subset \C$ and the Banach space $(A(K), \| \cdot \|_\infty)$ of con\-ti\-nuous functions $K \to \C$ that are holomorphic on $K^0$. Let $$T_{K,n} : f \in H(\D) \mapsto (S_nf)|_K \in A(K)\,\,\,(n \in \N ),$$ where $S_nf$ denotes the $n$th partial sum of the Taylor series of $f$ around the origin. Assume that $K \subset \C \setminus \D$ and that $K$ has connected complement. Then Costakis and Tsirivas \cite[Sect.~3]{costakistsirivas2014} have recently shown that, given any two strictly increasing sequences $(n_k),(m_k)$ in $\N$, the sequences $(T_{K,n_k})$ and $(T_{K,m_k})$ are --by using our terminology-- s-universal.
Even more, they have shown that $$\bigcap\,\big\{s\mbox{-}\mathcal U ((T_{K,n_k}),(T_{K,m_k})) : \, K \subset \C \setminus \D \mbox{ compact}, \, \C \setminus K \mbox{ connected}\big\}$$
\,is a residual subset of $H(\D )$.

\vspace{16pt}

\noindent {\bf Acknowledgements.} The first author has been partially supported by Plan
Andaluz de Investigaci\'on de la Junta de Andaluc\'{\i}a FQM-127
Grant P08-FQM-03543 and by MEC Grant MTM2015-65242-C2-1-P. The second author has been supported by DFG-Forschungsstipendium JU 3067/1-1.

%%%%%%%%%%
%REFERENCES %
%%%%%%%%%%

\end{document}